\documentclass[11pt]{amsart}

\newcommand{\arxivorjournal}[2]{#1}
\newcommand{\arxivonly}[1]{#1}
\newcommand{\journalonly}[1]{}

\usepackage[utf8]{inputenc}  

\arxivonly{
  \setlength{\textwidth}{\paperwidth}
  \addtolength{\textwidth}{-2.5in}
  \setlength{\textheight}{\paperheight}
  \addtolength{\textheight}{-1in}
  \calclayout
}

\arxivonly{
  \usepackage{amsthm}
}
\usepackage{amsfonts}
\usepackage{amsmath}
\usepackage{amssymb}
\usepackage{mathtools} 
\usepackage[mathscr]{eucal}

\arxivorjournal{
  \usepackage{xcolor}
  \definecolor{darkgreen}{rgb}{0,0.45,0}
  \definecolor{darkred}{rgb}{0.75,0,0}
  \definecolor{darkblue}{rgb}{0,0,0.6}
  \usepackage[colorlinks,citecolor=darkgreen,linkcolor=darkred,urlcolor=darkblue]{hyperref}
  \usepackage{breakurl}
  \usepackage{enumerate} 
}{
  \usepackage[colorlinks]{hyperref}
}

\usepackage{mathpartir}

\usepackage{graphicx}
\usepackage{tikz}
\usepackage{tikz-cd}

\usepackage{xifthen} 
\usepackage{filecontents}

\arxivorjournal{
  \theoremstyle{plain}
  \newtheorem{theorem}{Theorem}[section]
  \newtheorem{proposition}[theorem]{Proposition}
  \newtheorem{lemma}[theorem]{Lemma}
  \newtheorem{fact}[theorem]{Fact}
  \newtheorem{corollary}[theorem]{Corollary}
  \newtheorem{conjecture}[theorem]{Conjecture}
  
  \theoremstyle{definition}
  \newtheorem{definition}[theorem]{Definition}
  \newtheorem{remark}[theorem]{Remark}
  \newtheorem{notation}[theorem]{Notation}
  \newtheorem{example}[theorem]{Example}

  \newcommand{\thmqed}{\qed}
}{
  \newtheorem{theorem}{Theorem}[section]
  \newtheorem{proposition}[theorem]{Proposition}
  \newtheorem{lemma}[theorem]{Lemma}
  
  \newtheorem{corollary}[theorem]{Corollary}
  \newtheorem{conjecture}[theorem]{Conjecture}
  
  \newdefinition{definition}[theorem]{Definition}
  \newdefinition{remark}[theorem]{Remark}
  \newdefinition{notation}[theorem]{Notation}
  \newdefinition{example}[theorem]{Example}

  \newproof{proof}{Proof}
  
  \newcommand{\thmqed}{}
  \newcommand{\qedhere}{}
}

\makeatletter
\arxivonly{
  \renewcommand{\subsection}{\@startsection{subsection}{2}%
    \z@{.5\linespacing\@plus.7\linespacing}{.1\linespacing\@afterindentfalse}%
    {\normalfont\bfseries}}
}
\makeatother

\newcommand{\defemph}[1]{\textbf{#1}} 

\newcommand{\mystrut}[1]{\vrule height #1 depth 0pt width 0pt}

\usetikzlibrary{fit}
\usetikzlibrary{shapes}
\usetikzlibrary{arrows}
\usetikzlibrary{arrows}
\usetikzlibrary{decorations.markings}
\usetikzlibrary{positioning}
\usetikzlibrary{intersections}

\tikzset{
  commutative diagrams/arrow style=tikz,
  commutative diagrams/diagrams={row sep=large},
}

\tikzset{cd-style/.style={commutative diagrams/every diagram}}
\tikzset{cd-arrow-style/.style={commutative diagrams/.cd, every arrow, every label}}

\tikzset{fibtip/.tip={Triangle[open,angle=60:4.5pt]}}
\tikzset{tfibtip/.tip={Bar[sep]Triangle[open,angle=45:4pt]}} 

\pgfarrowsdeclarealias{c}{c}{right hook}{left hook}
\pgfarrowsdeclarealias{c'}{c'}{left hook}{right hook}

\tikzset{commutative diagrams/inj/.style={hook}}
\tikzset{commutative diagrams/tcof/.style={tail}}

\tikzset{fib/.code={\pgfsetarrowsend{fibtip}}}
\tikzset{tfib/.code={\pgfsetarrowsend{tfibtip}}}

\newcommand{\generalto}[3][1.5em]{ \mathrel{\mkern-1mu
  \tikz[baseline={([yshift=-0.55ex]a.south)}]{%
    \node[minimum width={#1},align=center,inner xsep=0.5ex,inner ysep=0.15ex] (a) {$\scriptstyle #3$};
    \draw[#2] (a.south west) -- (a.south east);}
 \mkern-1mu}}

\newcommand{\generalfrom}[3][1.5em]{ \mathrel{\mkern-1mu
  \tikz[baseline={([yshift=-0.55ex]a.south)}]{%
    \node[minimum width={#1},align=center,inner xsep=0.5ex,inner ysep=0.15ex] (a) {$\scriptstyle #3$};
    \draw[#2] (a.south east) -- (a.south west);}
  \mkern-1mu}}

\renewcommand{\to}[1][]{ \generalto{->}{#1} }
\newcommand{\shortto}[2][]{ \generalto[#2]{->}{#1} }
\newcommand{\from}[1][]{ \generalfrom{->}{#1} }
\newcommand{\fibto}[1][]{ \generalto{fib}{#1} }
\newcommand{\weqto}{ \generalto{->}{\sim} }

\newcommand{\parpair}[2]{\begin{tikzcd}[ampersand replacement=\&] #1 \ar[shift left]{r} \ar[shift right]{r} \& #2 \end{tikzcd}}

\newcommand{\A}{\mathbf{A}}
\renewcommand{\AA}{\vec A}
\newcommand{\B}{\mathbf{B}}

\newcommand{\C}{\mathbf{C}}
\newcommand{\catC}{\mathscr{C}}
\newcommand{\D}{\mathbf{D}}
\newcommand{\E}{\mathbf{E}}
\newcommand{\catE}{\mathscr{E}}
\newcommand{\ff}{\vec f}

\newcommand{\I}{\mathcal{I}}
\newcommand{\J}{\mathcal{J}}

\newcommand{\N}{\mathbb{N}}
\newcommand{\T}{\mathbf{T}}

\newcommand{\Y}{\mathbf{Y}}

\newcommand{\GGamma}{\vec \Gamma}
\newcommand{\DDelta}{\vec \Delta}

\newcommand{\Anod}{\mathcal{A}}
\newcommand{\Fib}{\mathcal{F}}
\newcommand{\TFib}{\mathcal{TF}}
\newcommand{\Cof}{\mathcal{C}}

\newcommand{\WEq}{\mathcal{W}}

\newcommand{\Cat}{\mathrm{Cat}}
\newcommand{\sSet}{\mathrm{sSet}}

\newcommand{\CxlCat}{\mathrm{CxlCat}}
\newcommand{\Lex}{\mathrm{Lex}}
\newcommand{\LCCC}{\mathrm{LCCC}}
\newcommand{\CwA}{\mathrm{CwA}}

\newcommand{\Cl}{\operatorname{Cl}}
\newcommand{\weCat}{\mathrm{weCat}}
\newcommand{\ElTop}{\mathrm{ElTop}}
\newcommand{\Set}{\mathrm{Set}}
\newcommand{\FibCat}{\mathrm{FibCat}}

\newcommand{\SpanCat}{\mathrm{Span}}
\newcommand{\EqvCat}{\mathrm{Eqv}}
\newcommand{\EqvReflCat}{\mathrm{EqvRefl}}
\newcommand{\EqvCompCat}{\mathrm{EqvComp}}
\newcommand{\EqvInvCat}{\mathrm{EqvInv}}

\newcommand{\supfunctor}[2]{\ifthenelse{\equal{#2}{}}%
  {(-)^\mathrm{#1}}%
  {{#2}^\mathrm{#1}}%
}
\newcommand{\Span}[1][]{\supfunctor{Span}{#1}}
\newcommand{\Eqv}[1][]{\supfunctor{Eqv}{#1}}
\newcommand{\EqvRefl}[1][]{\supfunctor{EqvRefl}{#1}}
\newcommand{\EqvComp}[1][]{\supfunctor{EqvComp}{#1}}
\newcommand{\EqvInv}[1][]{\supfunctor{EqvInv}{#1}}

\newcommand{\Nf}{\mathrm{N}_{\mathrm{f}}}

\DeclareMathOperator{\core}{core}
\DeclareMathOperator{\Hom}{Hom}

\newcommand{\op}{\mathrm{op}}
\DeclareMathOperator{\ob}{Ob}

\DeclareMathOperator{\ft}{ft}
\DeclareMathOperator{\Ho}{Ho}
\newcommand{\id}{\mathrm{id}}
\DeclareMathOperator{\ev}{ev}

\newcommand{\Imp}{\Rightarrow}
\newcommand{\Pmi}{\Leftarrow}
\newcommand{\iso}{\cong}
\renewcommand{\equiv}{\simeq}
\newcommand{\rhomot}{\sim_r}
\newcommand{\ideq}{=_{\Id}}
\newcommand{\wrhomot}{\sim}
\newcommand{\homot}{\sim}
\newcommand{\orthog}{\pitchfork}

\newcommand{\of}[1][1]{\mspace{#1 mu plus 1.0mu}\mathord{:}\mspace{#1 mu plus 1mu}}
\newcommand{\types}{\vdash}
\newcommand{\type}{\ \textsf{type}}
\newcommand{\cxt}{\ \textsf{cxt}}

\newcommand{\Id}{\mathsf{Id}}
\newcommand{\refl}{\mathsf{r}}
\newcommand{\unit}{\mathsf{1}}
\newcommand{\synSigma}{\mathsf{\Sigma}}
\newcommand{\synPi}{\mathsf{\Pi}}
\newcommand{\Piext}{\mathsf{\Pi}_{\mathsf{ext}}}
\newcommand{\HoTT}{\mathsf{HoTT}}

\newcommand{\Ty}{\mathrm{Ty}}
\newcommand{\Tm}{\mathrm{Tm}}

\newcommand{\lorth}[1]{{}^{\orthog}{#1}}
\newcommand{\rorth}[1]{{#1}^{\orthog}}

\newcommand{\fibslice}[2]{{#1}/\!/{#2}} 
\newcommand{\fibslicefunc}[2]{{#1}/\!/{#2}} 

\newcommand{\aangled}[1]{\langle\!\langle\, #1 \,\rangle\!\rangle} 

\newcommand{\freestrux}[3]{\aangled{#1}_{#2}^{#3}}

\newcommand{\freeCxlCat}[2][]{\freestrux{#2}{#1}{}}
\newcommand{\freeCwA}[2][]{\freestrux{#2}{#1}{\CwA}}

\title{The homotopy theory of type theories}

\arxivorjournal{
  \author[K.~Kapulkin]{Krzysztof Kapulkin}
  \address{Dept.\ of Mathematics\\The University of Western Ontario\\London, Ontario}
  
  \author[P.~LeF.~Lumsdaine]{Peter LeFanu Lumsdaine}
  \address{Dept.\ of Mathematics\\Stockholm University\\Stockholm, Sweden}
}{
  \author[uwo]{Krzysztof Kapulkin}
  \address[uwo]{Dept.\ of Mathematics, The University of Western Ontario, London, Ontario}

  \author[sthlm]{Peter LeFanu Lumsdaine}
  \address[sthlm]{Dept.\ of Mathematics, Stockholm University, Stockholm, Sweden}
}

\date{August 7, 2018}

\begin{document}

\begin{abstract}
  We construct a left semi-model structure on the category of intensional type theories (precisely, on $\CxlCat_{\Id,\unit,\synSigma(,\Piext)}$).
  This presents an $\infty$-category of such type theories; we show moreover that there is an $\infty$-functor $\Cl_\infty$ from there to the $\infty$-category of suitably structured quasi-categories.

  This allows a precise formulation of the conjectures that intensional type theory gives internal languages for higher categories, and provides a framework and toolbox for further progress on these conjectures.
\end{abstract}

\maketitle


\section{Introduction}

Homotopy Type Theory (HoTT) arises from the discovery that the logical system of dependent type theory can be naturally interpreted in homotopy-theoretic settings, and provides a rich language for such settings, thanks largely to its richer treatment of equality compared to first-order logic.

Classically, equality carries no information beyond its truth value; but in type theory, objects can be equal/identifiable in a variety of ways, so a type of equalities may be a non-trivial type in its own right, analogous to the path space of a topological space.
This forms the basis of “synthetic homotopy theory”, a major direction within HoTT: the development of homotopy-theoretic constructions and theorems, entirely elementarily within type theory.
%

Of course, one wants to know that these can be interpreted in a good range of established settings that a homotopy theorist might care about.
Conversely, one may hope that all homotopy-theoretic concepts (in some sense) can be translated into type theory.
Various results in these directions have been given --- some proven, some conjectured, some only informally sketched (see e.g.\ \cite{joyal:remarks-on-homotopical-logic,shulman:internal-languages,shulman:inverse-diagrams,shulman:elegant-presheaves,shulman:ei-diagrams,kapulkin:locally-cartesian-qcat}).

Such hopes have been summarised as the idea that HoTT should be the \defemph{internal language of $\infty$-categories}.
Precisely, by analogy with established “internal languages” in settings such as topos theory, this should mean a single master statement subsuming the above results: the existence of a suitable equivalence between some (higher) categories of type theories and $\infty$-categories.
%

The first contribution of the present paper is a framework for formulating such a claim precisely.
We do so by assembling type theories into a higher category, and giving a functor $\Cl_\infty$ from this to a higher category of suitably structured quasicategories.
The internal language conjecture can then be stated as: $\Cl_\infty$ is an equivalence of higher categories.

The other main contribution is a left semi-model structure on the category of type theories.
This gives a tractable and explicit presentation of the higher category thereof, which we hope will provide a solid base for further progress on the conjectures.

In a little more detail: we work with “type theories” as contextual categories or categories with attributes (CwA’s), keeping our results independent of the correspondence between these and syntactically presented theories.
We assume $\Id$\nobreakdash-, $\Sigma$\nobreakdash-, and unit types throughout; we consider also the extension to $\Pi$-types.

Two technical tools of the paper may be of independent interest.
One (small, but useful and to our knowledge new) is a notion of equivalence between arbitrary objects of a CwA.
The other is the construction of the CwA of span-equivalences in a given CwA, a powerful tool for constructing equivalences between CwA’s.

In Section \ref{sec:spans}, we make use of some results from our forthcoming article \cite{kapulkin-lumsdaine:inverse-diagrams}, currently in preparation.
However, those results may be treated as black boxes; the present paper can be read as essentially self-contained.

During the preparation of this paper, we learned that Valery Isaev has independently given a similar construction in \cite{isaev:model-structure}, defining a (full) model structure on a slightly different category of type theories (assuming an interval type, instead of Martin-Löf identity types).


\section{Background}

In this section, we review the necessary background on categorical models of type theory. We recall the definition of a contextual category and introduce the notation for working with them. We then investigate their homotopy-theoretic properties, employing the language of fibration categories.

\subsection{Contextual categories and functors}

We choose to work with contextual categories as our model of type theory. These were introduced by Cartmell in his thesis \cite{cartmell:thesis} and studied by Streicher \cite{streicher:semantics-book} and more recently in a series of papers by Voevodsky (see e.g., \cite{voevodsky:C-system-from-universe,voevodsky:quotients-C-system,voevodsky:products,voevodsky:identity-types}).

\begin{definition} \label{def:cxl-cat}
  A \defemph{contextual category} $\C$ consists of the following data:
  \begin{enumerate}
  \item a category $\C$;
  \item a grading of objects as $\ob \C = \coprod_{n : \N} \ob_n \C$;
  \item an object $1 \in \ob_0 \C$;
  \item \defemph{father} operations $\ft_n : \ob_{n+1} \C \to \ob_n \C$ (whose subscripts we suppress);
  \item for each $\Gamma \in \ob_{n+1} \C$, a map $p_\Gamma : \Gamma \to \ft \Gamma$ (the \defemph{canonical projection} from $\Gamma$, distinguished in diagrams as $\fibto$);
  \item for each $\Gamma \in \ob_{n+1} \C$ and $f : \Delta \to \ft \Gamma$, an object $f^* \Gamma$ together with a \defemph{connecting map} $f.\Gamma : f^*\Gamma \to \Gamma$;
    \newcounter{tempcounter}
    \setcounter{tempcounter}{\theenumi}
  \end{enumerate}  
  such that:
  \begin{enumerate}
    \setcounter{enumi}{\thetempcounter}
  \item $1$ is the unique object in $\ob_0 \C$;
  \item $1$ is a terminal object in $\C$;
  \item for each $\Gamma \in \ob_{n+1} \C$, and $f : \Delta \to \ft \Gamma$, we have $\ft(f^*\Gamma) = \Delta$, and the square
    \[\begin{tikzcd}
      f^* \Gamma \ar[d, fib, "p_{f^*\Gamma}"'] \ar[r,"{f.\Gamma}"] \arrow[dr, phantom, "\lrcorner", very near start] & \Gamma \ar[d, fib, "p_\Gamma"] \\
      \Delta                   \ar[r,"f"]      & \ft \Gamma
    \end{tikzcd}\]
    is a pullback (the \defemph{canonical pullback} of $\Gamma$ along $f$); and
  \item these canonical pullbacks are strictly functorial: that is, for $\Gamma \in \ob_{n+1} \C$, $\id_{\ft \Gamma}^* \Gamma = \Gamma$ and $\id_{\ft \Gamma}.\Gamma = \id_\Gamma$; and for $\Gamma \in \ob_{n+1} \C$, $f : \Delta \to \ft \Gamma$ and $g : \Theta \to \Delta$, we have $(fg)^* \Gamma = g^* f^* \Gamma$ and $fg.\Gamma = f.\Gamma \circ g.f^*\Gamma$.
  \end{enumerate}
\end{definition}

Contextual categories can be easily seen as models of an essentially algebraic theory with sorts indexed by $\N + \N \times \N$.
As such, they come with a canonical notion of morphism: a \defemph{contextual functor} $F : \C \to \D$ between contextual categories is a homomorphism between them, regarded as models of an essentially algebraic theory.
Explicitly, $F$ is a functor preserving on the nose all the structure of Definition \ref{def:cxl-cat}: the grading on objects, the terminal object, the father maps, the dependent projections, the canonical pullbacks, and the connecting maps.

We denote the category of contextual categories and contextual functors by $\CxlCat$.

\begin{notation} \label{notation:cxl-cats}
  Given $\Gamma \in \ob_n \C$, we write $\Ty_\C(\Gamma)$ for the set of objects $\Gamma' \in \ob_{n+1} \C$ such that $\ft(\Gamma') = \Gamma$, and call these \defemph{types in context $\Gamma$}.
  For $A \in \Ty_\C (\Gamma)$, we write $\Gamma.A$ for $A$ considered as an object of $\C$, $p_A$ for the projection $p_{\Gamma.A} : \Gamma.A \to \Gamma$, and $f.A : f^*(\Gamma.A) \to \Gamma.A$ for the connecting map $f.(\Gamma.A)$. 
  For each $f : \Gamma' \to \Gamma$, we have a map $f^* : \Ty_\C (\Gamma) \to \Ty_\C (\Gamma')$ given by the pullback operation of $\C$.
  The axioms of a contextual category ensure that this forms a presheaf $\Ty_\C : \C^\op \to \Set$.

  More generally, by a \defemph{context extension} of $\Gamma \in \ob_n \C$, we mean some object $\Gamma' \in \ob_{n+m} \C$ with $\ft^m \Gamma' = \Gamma$.
  Again, we will write such an extension (considered as an object of $\C$) as $\Gamma.\Delta$, with a canonical projection $p_\Delta : \Gamma.\Delta \to \Gamma$ obtained by composing the projections $\Gamma.\Delta \to \ft(\Gamma.\Delta) \to \ldots \to \ft^m(\Gamma.\Delta) = \Gamma$.
  Similarly, given $f : \Gamma' \to \Gamma$ and a context extension $\Delta$ of $\Gamma$, by iterating the pullback of types we obtain a pullback context extension $f^*\Delta$ over $\Gamma'$, with $f.\Delta : \Gamma'.f^*\Delta \to \Gamma.\Delta$.

  Given $\Gamma \in \C$ and $A \in \Ty_\C (\Gamma)$, we write $\Tm_{\C,\Gamma}(A)$ for the set of sections $s : \Gamma \to \Gamma.A$ of the projection $p_A$.
  When no confusion is possible, we will omit the subscripts, writing $\Ty(\Gamma)$ and $\Tm(A)$ respectively.
\end{notation}

\begin{definition}
  Contextual categories can be equipped with additional operations corresponding to the various type-constructors of Martin-L\"of Type Theory.
  For the present paper we consider just the structure corresponding to:
  \begin{enumerate}
  \item identity types (denoted $\Id$);
  \item unit types ($\unit$) and dependent sum types ($\synSigma$);
  \item dependent function types, with functional extensionality rules (together, $\Piext$).
  \end{enumerate}
  For the definitions of these structures, see \cite[App.\ B]{kapulkin-lumsdaine:simplicial-model}.

  For each choice of constructors, contextual categories with such structure are again models of an essentially algebraic theory (extending the e.a.t.\ of contextual categories), so have a natural notion of morphism: contextual functors preserving the extra structure.
 
  We write $\CxlCat_{\Id,\unit,\synSigma}$ for the category of contextual categories equipped with $\Id$\nobreakdash-, $\unit$\nobreakdash-, and $\synSigma$-types; and $\CxlCat_{\Id,\unit,\synSigma,\Piext}$ the category of contextual categories with all these plus extensional dependent function types.
  When a statement, construction, or proof can be read in parallel for each of these categories, we will refer to them as $\CxlCat_{\Id,\unit,\synSigma(,\Piext)}$.
  There is moreover (again, by their description as e.a.t.’s) a free--forgetful adjunction 
  \[ \begin{tikzcd} \CxlCat_{\Id,\unit,\synSigma}
      \ar[rr,bend left=5] \ar[rr,phantom,"\scriptstyle \bot"]
    & & \ar[ll,bend left=5] \CxlCat_{\Id,\unit,\synSigma\mathrlap{,\Piext}}  \end{tikzcd} \]
\end{definition}

\begin{remark}
  One often considers other logical structure besides $\Id$, $\unit$, $\synSigma$, and $\Piext$.
  Some of the results of this paper extend directly to such further structure; others do not.
  In the absence of a good general framework for such structure, however, we restrict ourselves for the present paper to the case of $\Id, \unit, \synSigma (, \Piext)$, except for a few definitions and constructions that only assume $\Id$-types.
\end{remark}

\begin{definition}
  Following Garner \cite[Prop.\ 3.3.1]{garner:2-d-models}, we note that $\Id$-types on a contextual category allow the construction of more general \defemph{identity contexts}.
  Specifically, given $\Gamma \in \ob_n \C$ and a context extension $\Gamma.\Delta \in \ob_{n+m} \C$, there is a further context extension $\Gamma.\Delta.p_{\Delta}^*\Delta.\Id_\Delta \in \ob_{n+3m}\C$, along with a reflexivity map and elimination operation generalizing those of the identity type $\Gamma.A.p_A^*A.\Id_A$ of a single type over $\Gamma$.
\end{definition}

\begin{definition}\label{def:homotopy-in-cxl-cat}
  Given $f, g : \Gamma \to \Delta$ in a contextual category $\C$, a \defemph{homotopy} $H$ from $f$ to $g$ (denoted $H : f \homot g$) is a factorization of $(f,g) : \Gamma \to \Delta \times \Delta = \Delta . p_{\Delta}^*\Delta$ through the identity context $\Delta.p_{\Delta}^*\Delta.\Id_\Delta \to[p_{\Id_\Delta}] \Delta . p_{\Delta}^*\Delta$.
\end{definition}

There are various established definitions of equivalence in contextual categories, all essentially equivalent \cite[Ch.\ 4]{hott:book}; we choose the following:
\begin{definition}\label{def:equiv-in-cxl-cat}
  Let $\C$ be a contextual category with identity types.
  
  A \defemph{structured equivalence} $w : \Gamma \equiv \Delta$ consists of a map $f : \Gamma \to \Delta$, together with maps $g_1, g_2 : \Delta \to \Gamma$ and homotopies $\eta : fg_1 \homot 1_\Delta$ and $\varepsilon : g_2f \homot 1_\Gamma$.

  An \defemph{equivalence} $\Gamma \weqto \Delta$ in $\C$ is a map $f : \Gamma \to \Delta$ for which there exist some $g_1, g_2, \eta, \varepsilon$ making it a structured equivalence.
\end{definition}

\begin{definition}\label{def:fibrant-slice-of-cxl-cat}
  Given a contextual category $\C$ and an object $\Gamma \in \C$, the \defemph{fibrant slice} contextual category $\fibslice{\C}{\Gamma}$ is given by:
 \begin{enumerate}
  \item objects in $\ob_m \fibslice{\C}{\Gamma}$ are context extensions $\Gamma.\Delta \in \ob_{n+m} \C$;
  \item $(\fibslice{\C}{\Gamma})(\Gamma.\Delta, \ \Gamma.\Delta') := (\C/\Gamma)(\Gamma.\Delta, \ \Gamma.\Delta')$;
  \item the remaining structure is inherited from $\C$.
 \end{enumerate}

  If $\C$ carries identity types (resp.\ $\unit$, $\synSigma$, $\Piext$), then so does $\fibslice{\C}{\Gamma}$.
\end{definition}

This satisfies the familiar categorical property that a slice of a slice is again a slice, in that $\fibslice{(\fibslice{\C}{\Gamma})}{\Gamma.\Delta} \iso \fibslice{\C}{(\Gamma.\Delta)}$.

Moreover, any contextual functor $F : \C \to \D$ and object $\Gamma \in \C$ induce an evident contextual functor $\fibslicefunc{F}{\Gamma} : \fibslice{\C}{\Gamma} \to \fibslice{\D}{F\Gamma}$; and this preserves any logical structure that $F$ does.

\subsection{Fibration categories} 

Fibration categories and variations thereof (like Shulman's \defemph{type-theoretic fibration categories} \cite[Def.\ 2.1]{shulman:inverse-diagrams} and Joyal's \defemph{tribes} \cite{joyal:tribes}) have proven useful when studying homotopy-theoretic aspects of type theory.
We begin by recalling some of the basic definitions and constructions.

Fibration categories were introduced by Brown \cite{brown:abstract-homotopy-theory} as \defemph{categories of fibrant objects}.
We slightly strengthen them, following other recent authors (e.g.\ Szumi\l{}o \cite[Def.\ 1.1]{szumilo:two-models}).

\begin{definition} \label{def:fibration-category}
 A \defemph{fibration category} consists of a category $\catC$ together with two wide subcategories (subcategories containing all objects): $\Fib$ of \defemph{fibrations} and $\WEq$ of \defemph{weak equivalences} such that:
 \begin{enumerate}
  \item weak equivalences satisfy the \defemph{2-out-of-6} property; that is, given a composable triple of morphisms
  \[ X \to[f] Y \to[g] Z \to[h] Z, \]
  if $hg$ and $gf$ are weak equivalences, then so are $f$, $g$, $h$, and $hgf$.
  \item all isomorphisms are \defemph{acyclic fibrations} (i.e., are both fibrations and weak equivalences).
  \item pullbacks along fibrations exist; fibrations and acyclic fibrations are stable under pullback.
  \item $\catC$ has a terminal object $1$; the canonical map $X \to 1$ is a fibration for any object $X \in \catC$ (that is, all objects are \defemph{fibrant}).
  \item every map can be factored as a weak equivalence followed by a fibration.
 \end{enumerate}
\end{definition}

Given a fibration category $\catC$, its \defemph{homotopy category} $\Ho\catC$ is the result of formally inverting the weak equivalences. 
It can be described more explicitly using the notion of weak right homotopy.

\begin{definition}[{\cite[\textsection 2]{brown:abstract-homotopy-theory}}] \label{def:homotopy-in-fib-cat}
  A \defemph{path object} for $X \in \catC$ is any factorization $X \weqto PX \fibto X \times X$ of the diagonal map as a weak equivalence followed by a fibration.
  
  Maps $f, g : X \to Y$ are \defemph{weakly right homotopic}, $ f \wrhomot g$, if for some trivial fibration $t : X' \to X$ the maps $ft,gt : X' \to Y$ factor jointly through some path object $PY \fibto Y \times Y$. 
  \[\begin{tikzcd}
    X' \ar[d, tfib,"t"] \ar[r,"h"] & PY \ar[d, fib] \\
    X \ar[r, "{\langle f,g \rangle}"] & Y \times Y.
  \end{tikzcd}\]
  Say $f, g$ are \defemph{(strictly) right homotopic}, $f \rhomot g$, if one can take $X' = X$, $t = \id_X$.%
\footnote{In some recent literature, e.g.\ Szumi\l{}o \cite{szumilo:two-models}, \emph{right homotopic} is used for the weak notion; we distinguish that explicitly to avoid clashing with more established usage.}
\end{definition}

In general fibration categories, the weak notion is more important:
 
\begin{theorem}[{\cite[Thm.~1]{brown:abstract-homotopy-theory}}] \label{thm:ho-fib-cat}
  For any fibration category $\catC$, the homotopy category $\Ho\catC$ may be taken as the category with the same objects as $\catC$, and with $\Hom_{\Ho\catC}(X, Y) = \Hom_{\catC}(X, Y)/\homot$.
\end{theorem}

\begin{definition}
  A functor between fibration categories is \defemph{exact} if it preserves fibrations, acyclic fibrations, pullbacks along fibrations, and a terminal object.

  An exact functor is a \defemph{weak equivalence} of fibration categories if it induces an equivalence of homotopy categories.
\end{definition}

As mentioned above, the framework of fibration categories can be used to study homotopy-theoretic aspects of type theory.
Let $\C$ be a contextual category with $\Id$-types.
Define classes $\WEq$, $\Fib$ of maps in $\C$ by:
\begin{itemize}
 \item $\WEq$ is precisely the equivalences of Definition \ref{def:equiv-in-cxl-cat};
 \item $\Fib$ consists of maps isomorphic to some composite of canonical projections.
\end{itemize}

\begin{theorem}[{\cite[Thm.\ 3.2.5]{avigad-kapulkin-lumsdaine}, \cite[Ex.\ 2.6.(3)]{kapulkin:locally-cartesian-qcat}}] \label{thm:akl} \leavevmode
\begin{enumerate}
 \item For any contextual category $\C$ with $\Id$, $\unit$, and $\synSigma$, these classes $\Fib$ and $\WEq$ make $\C$ a fibration category.
 \item This forms the object part of a faithful functor $\CxlCat_{\Id, \unit, \synSigma} \to \FibCat$.
 \end{enumerate}
\end{theorem}

The fibration categories arising from contextual categories are particularly nice in that all of their objects are also cofibrant; that is, every acyclic fibration admits a section \cite[Lem.\ 3.2.14]{avigad-kapulkin-lumsdaine}.
This justifies the following description of their homotopy categories:

\begin{lemma} \label{lem:homot-cat-of-cxl-cat}
  Let $\C \in \CxlCat_{\Id,\unit,\synSigma}$.
  Then the homotopy category of $\C$ (regarded as a fibration category) can be described as follows:
  \begin{enumerate}
  \item objects of $\Ho\C$ are the objects of $\C$;
  \item morphisms $\Gamma \to \Delta$ in $\Ho\C$ are homotopy classes of maps in $\C$, in the sense of Definition \ref{def:homotopy-in-cxl-cat}.
  \end{enumerate}
\end{lemma}

\begin{proof}
  As shown in \cite[Thm.\ 3.2.5]{avigad-kapulkin-lumsdaine}, the path-objects in a contextual category $\C$ are given exactly by the identity contexts.
  Homotopy in the sense of Definition \ref{def:homotopy-in-cxl-cat} is therefore exactly right homotopy in the sense of Definition \ref{def:homotopy-in-fib-cat}.
  Furthermore, since every object is cofibrant, this coincides with weak right homotopy.
\end{proof}

Finally, we note an indispensible (and easily overlooked) lemma: the property of being an equivalence does not depend on where one views a map.

\begin{lemma} \label{lem:equiv-in-slice}
  Let $\C$ be a contextual category with identity types.
  A map $f : \Gamma.\Delta \to \Gamma.\Delta'$ over $\Gamma$ is an equivalence in $\fibslice{\C}{\Gamma}$ if and only if it is an equivalence in $\C$.
\end{lemma}

\begin{proof}
  More generally, let $f : Y \to Y'$ be a map of fibrations over a base $X$, in any fibration category $\catC$, with $Y$ and $Y'$ cofibrant.
  Then $f$ is a homotopy equivalence in $\catC$ if and only if it is one in $\catC/X$.
  This follows by an argument originally due to Dold; it is given for model categories in \cite[Thm.~6.3]{kamps-porter:abstract-and-simple}, but adapts directly to the present setting.
\end{proof}


\section{Equivalences, fibrations, and cofibrations of contextual categories} \label{sec:classes-of-maps}

In this section, we will define the three classes of maps in $\CxlCat_{\Id,\unit,\synSigma(,\Piext)}$: weak equivalences, cofibrations, and fibrations required for the left semi-model structure, as well as state the internal language conjectures.

We begin by introducing two notions of equivalence between contextual categories: type-theoretic and homotopy-theoretic, and proving that they are equivalent (Proposition \ref{prop:characterization-of-equivalences}).
We then review the basic facts about known connections between type theory and higher category theory, and state the internal language conjectures (\ref{conj:internal-languages}).
In the remainder of the section, we introduce notions of (trivial) fibrations and cofibrations between contextual categories, proving some their properties.

\subsection{Logical and homotopy-theoretic equivalences}

\begin{definition} \label{def:equiv-of-cxl-cats}
  A map $F : \C \to \D$ of contextual categories with $\Id$-types is a \defemph{(type-theoretic) equivalence} if it satisfies
  \begin{enumerate}
  \item \defemph{weak type lifting}: for any $\Gamma \in \C$ and $A \in \Ty(F\Gamma)$, there exists $\bar{A} \in \Ty(\Gamma)$ together with an equivalence $F\bar{A} \weqto A$ over $F\Gamma$; and
  \item \defemph{weak term lifting}: for any $\Gamma \in \C$, $A \in \Ty(\Gamma)$, and $a \in \Tm(FA)$, there exists $\bar{a} \in \Tm(A)$ together with an element of the identity type $e \in \Tm(\Id_{FA}(F\bar{a},a))$.
  \end{enumerate}
  Write $\WEq$ for the class of type-theoretic equivalences in $\CxlCat_{\Id,\unit,\synSigma(,\Piext)}$. 
\end{definition}

From a logical perspective, this is a sort of conservativity between theories: compare e.g.\ the condition \textsf{TY-CONS} of \cite[\textsection 3.2.3]{hofmann:thesis}.

Both lifting properties can in fact be strengthened:

\begin{lemma}\label{lem:cxt-sec-lifting}
  Every type-theoretic equivalence $F : \C \weqto \D$ additionally satisfies
  \begin{enumerate}
  \item \defemph{weak context lifting}: for any context $\Gamma \in \C$ and context extension $F\Gamma.\Delta$, there exists a context extension $\Gamma.\bar{\Delta}$ together with an equivalence $F(\Gamma.\bar{\Delta}) \weqto F\Gamma.\Delta$ over $F\Gamma$; and
  \item \defemph{weak section lifting}: for any context extension $\Gamma.\Delta \in \C$ and section $s : F\Gamma \to F\Gamma.F\Delta$ of the generalized projection $p_{F\Delta}$, there exists a section $\bar{s} : \Gamma \to \Gamma.\Delta$ of $p_\Delta$, together with a homotopy $e : F \bar{s} \homot s$ over $F \Gamma$.
  \end{enumerate} 
\end{lemma}

\begin{proof}
 By definition of identity context and induction on the length of the context.
\end{proof}

One can also use the associated fibration category (Theorem \ref{thm:akl}) to define equivalences in a more homotopy-theoretic way.
Specifically, call a map $F : \C \to \D$ a \defemph{homotopy-theoretic equivalence} if the induced functor $\Ho F : \Ho \C \to \Ho \D$ is an equivalence of categories.
It turns out these two definitions coincide:

\begin{proposition} \label{prop:characterization-of-equivalences}
 A contextual functor is a type-theoretic equivalence if and only if it is a homotopy-theoretic equivalence.
\end{proposition}

\begin{proof}
  We rely on Lemma~
  \ref{lem:homot-cat-of-cxl-cat} throughout.
  
  First, assume $F : \C \to \D$ is a type-theoretic equivalence.
  Then $\Ho F$ is:
  \begin{enumerate}
  \item full, by weak section lifting (Lemma \ref{lem:cxt-sec-lifting}(2)): a map $f : F\Gamma \to F\Delta$ can be viewed as a section of the context extension $F\Gamma.p_{F \Gamma}^* F\Delta \to F\Gamma$ and as such can be lifted (up to homotopy) to a map $\bar{f} : \Gamma \to \Gamma.\Delta$;
  \item faithful, by weak section lifting applied to pullbacks of identity contexts;
  \item essentially surjective, by weak context lifting (Lemma \ref{lem:cxt-sec-lifting}(1)).
  \end{enumerate}
  
  Conversely, assume that $\Ho F$ is an equivalence of categories.
  For weak type lifting, suppose $A \in \Ty(F \Gamma)$.
  Since $\Ho F$ is essentially surjective, one can find $\Gamma' \in \C$ and $w : F\Gamma' \weqto F\Gamma.A$.
  Moreover, since $\Ho F$ is full, there is some $f : \Gamma' \to \Gamma$ such that the triangle
  \[\begin{tikzcd}
    F\Gamma' \ar[rr, "w", "\sim"'] \ar[rd, "Ff"'] & & F\Gamma.A \ar[ld,fib,"p_A"] \\ 
    & F\Gamma &
   \end{tikzcd}\] 
   commutes up to homotopy.
   Factoring $f$ as an equivalence $u$ followed by a fibration $p_\Delta : \Gamma.\Delta \to \Gamma$, and taking the iterated $\synSigma$-type of $\Delta$, we obtain $\bar{A} \in \Ty(\Gamma)$ and a triangle commuting up to homotopy
   \[\begin{tikzcd}
    F\Gamma.F\bar{A} \ar[rr, "w.Fu^{-1}", "\sim"'] \ar[rd, "Fp_{\bar{A}}"'] & & F\Gamma.A \ar[ld,fib,"p_A"] \\ 
    & F\Gamma &
   \end{tikzcd}\]
   Now since $p_A$ is a fibration we can replace $w \cdot Fu^{-1}$ by some homotopic map strictly over $F \Gamma$, as required.

  Lastly, we give weak term lifting; this is a little more involved.
  We start by showing a “crude section lifting” property: for any $\Gamma.\Delta \in \C$ and section $a : F\Gamma \to F(\Gamma.\Delta)$ of $p_{F\Delta}$, there is some section $\widehat{a} : \Gamma \to \Gamma.\Delta$ of $p_\Delta$, together with a homotopy $h : F\widehat{a} \homot a$ (but not yet necessarily over $F\Gamma$, as required in weak term/section lifting) .
  
  Given such $\Gamma.\Delta$ and $a$, by fullness of $\Ho F$ there is some map $a' : \Gamma \to \Gamma.\Delta$ with $F a' \homot a$.
  Now $F(p_{\Delta} a') \homot p_{F\Delta} a = \id_{F \Gamma}$, so by faithfulness of $\Ho F$, $p_\Delta a' \homot \id_{\Gamma}$.
  So since $p_{\Delta}$ is a fibration, we can replace $a'$ by some section $\widehat{a}$ of $p_\Delta$, with $\widehat{a} \homot a'$ and hence $F\widehat{a} \homot a$ as required.

  Now, we can strengthen this to full term-lifting.
  Given $\Gamma$, $A$, $a$, take by crude section-lifting some section $\widehat{a} : \Gamma \to \Gamma.A$ and homotopy $h : F\widehat{a} \homot a : F \Gamma \to F\Gamma.FA$.
  We can split $h$ up into $h_0 = p_{F A}h : \id_{F\Gamma} \homot \id_{F \Gamma}$ and $h_1 \in \Tm_{F \Gamma}(\Id_{FA}((h_0)_! F\widehat{a},a))$, where $(h_0)_!$ denotes transport (as used in the definition of identity contexts).
  In type-theoretic notation,
 \begin{align*}
  x \of F \Gamma & \types h_0(x) : \Id_{F\Gamma} (x,x), \\
  x \of F \Gamma & \types h_1(x) : \Id_{FA(x)}(h_0(x)_! F\widehat{a}(x), a(x)).
 \end{align*}

  We next apply crude section lifting to $h_0 : F \Gamma \to F \Gamma . \Id_{F \Gamma}(\id_{F\Gamma},\id_{F\Gamma})$.
  Written type-theoretically, the result is a term $x \of \Gamma \types \widehat{h_0}(x) : \Id_\Gamma (x,x)$, together with a homotopy $\alpha : F\widehat{h_0} \homot h_0$, which once again we split up into two parts:
 \begin{align*}
  x \of F \Gamma &\types \alpha_0(x) : \Id_{F\Gamma} (x,x), \\
  x \of F \Gamma & \types \alpha_1(x) : \Id_{\Id_{F\Gamma}(x,x)}(\alpha_0(x)_! F\widehat{h_0}(x), h_0(x)).
 \end{align*}

  %
  By the $J$-structure ($\Id$-elimination), one can produce a term $x, y \of F \Gamma, u \of \Id_{F \Gamma}(x,y) \types \theta(x,y,u) : \Id_{\Id(y,y)}( u_! F\widehat{h_0}(x), F\widehat{h_0}(y))$.%
  \footnote{In traditional homotopy-theoretic terms: $F \widehat{h_0}$ is a global section of the free loop space of $F \Gamma$, so must land in the \emph{center} of $\pi_0(F\Gamma)$.}
  So, in particular, we get a term $x \of F \Gamma \types \beta(x) : \Id_{\Id_{F\Gamma}(x,x)}(F\widehat{h_0}(x), h_0(x))$, by composing $\theta(x,x,\alpha_0(x))^{-1}$ with $\alpha_1(x)$.

  Now define the corrected lifting of $a$ as $x \of \Gamma \types \bar{a}(x) := \widehat{h_0}(x)_! \widehat{a}(x) : A (x)$.
  The desired equality $x \of F\Gamma \types \Id_{FA} (F\bar{a}(x), a(x))$ is then a composite:
\[ F(\bar{a}(x)) =  F(\widehat{h_0}(x))_! F(\widehat{a}(x)) \overset{\beta(x)_!}{\ideq} h_0(x)_! F(\widehat{a}(x)) \overset{h_1(x)}{\ideq} a(x). \qedhere \]
\end{proof}

This proposition justifies dropping the distinction, and simply calling such functors \emph{equivalences}.
It also immediately gives:

\begin{corollary} \label{cor:2-of-6-cxl-cats}
 The class of equivalences satisfies 2-out-of-6 and is closed under retracts.
\end{corollary}

\begin{proof}
 Equivalences of categories are closed under 2-out-of-6 and retracts; thus so are equivalences of contextual categories, as their inverse image under $\Ho$ by Proposition~\ref{prop:characterization-of-equivalences}.
\end{proof}

Closing, we note another useful property:

\begin{lemma} \label{lem:fibslicefunc-equiv}
 If a contextual functor $F : \C \to \D$ is an equivalence, then so is the induced functor on slices $\fibslicefunc{F}{\Gamma} : \fibslice{\C}{\Gamma} \to \fibslice{\D}{F\Gamma}$.
\end{lemma}

\begin{proof}
  Straightforward, with the use of Lemma~\ref{lem:equiv-in-slice}.
\end{proof}

\subsection{Conjectures on internal languages}

In this section, we will provide precise statements of the conjectures establishing dependent type theories as internal languages of (sufficiently structured) higher categories. Thus its goal is to put the results of the remainder of the paper in a broader context.

By a \defemph{category with weak equivalences}, we mean a pair $(\catC,\WEq)$, where $\catC$ is a category and $\WEq$ a wide subcategory of $\catC$ (whose maps we call weak equivalences).
A functor $F$  between categories with weak equivalences $(\catC,\WEq)$, $(\catC',\WEq')$ is \defemph{homotopical} if it preserves weak equivalences.
Write $\weCat$ for the category of categories with weak equivalences and homotopical functors.

Every fibration category $(\catC, \Fib, \WEq)$ has an obvious underlying category with weak equivalences $(\catC, \WEq)$, and by Ken Brown's Lemma \cite[Lem.\ 1.1.12]{hovey:book}, every exact functor is homotopical.
It follows by Theorem \ref{thm:akl} that every contextual category has an underlying category with weak equivalences, and this construction forms a functor $\CxlCat_{\Id,\unit,\synSigma} \to \weCat$.

The category $\weCat$ can itself be regarded as a category with weak equivalences, where the weak equivalences are Dwyer--Kan equivalences (DK-equivalences) \cite{barwick-kan:characterization}.

Let $\Cat_\infty$ denote the full subcategory of the category $\sSet$ of simplicial sets, whose objects are quasicategories \cite[Def.\ 1.5]{joyal:notes-on-quasicategories}. We will consider $\Cat_\infty$ as a category with weak equivalences in which the weak equivalences are categorical equivalences \cite[Def.\ 1.20]{joyal:notes-on-quasicategories}.

The categories with weak equivalences $\weCat$ and $\Cat_\infty$ are DK-equivalent and we will write $\Ho_\infty$ for an equivalence $\weCat \to \Cat_\infty$.
While this functor may be implemented in many ways (see \cite[\textsection 1.6]{barwick:k-theory-of-higher-cats} for several possibilities), they are all equivalent, by \cite[Thm.\ 6.3]{toen:unicity}.
For concreteness, we take $\Ho_\infty$ to be the composite of the hammock localization followed by fibrant replacement and the homotopy coherent nerve.

We will write $\Cl^{\Id,\unit,\synSigma}_\infty$ for the composite functor
\[ \CxlCat_{\Id,\unit,\synSigma} \to \weCat \to[\Ho_\infty] \Cat_\infty \]
and $\Cl^{\Id,\unit,\synSigma,\Piext}$ for its composite with the forgetful functor $\CxlCat_{\Id,\unit,\synSigma, \Piext} \arxivorjournal{\to}{\shortto{0.9em}} \CxlCat_{\Id,\unit,\synSigma}$.

For working with these functors in practice, one can exploit Szumi{\l}o's construction of the \defemph{quasicategory of frames} $\Nf\catC$ in a fibration category $\catC$ \cite[\textsection 3.1]{szumilo:two-models}.
For any fibration category, there is an equivalence $\Nf\catC \equiv \Ho_\infty \catC$ \cite[Cor.\ 4.15]{kapulkin-szumilo}; so for fibration categories, $\Nf\catC$ gives a much more explicit and workable description of the quasicategory $\Ho_\infty\catC$ than is provided by the more general constructions.

\begin{theorem} \leavevmode
\begin{enumerate}
 \item \label{no:Cl-values-Lex}  The functor $\Cl^{\Id,\unit,\synSigma}_\infty$ takes values in the category $\Lex_\infty$ of quasicategories with finite limits and finite limit preserving functors. Moreover, it takes equivalences of contextual categories to categorical equivalences of quasicategories.
 \item \label{no:Cl-values-LCCC} The functor $\Cl^{\Id,\unit,\synSigma, \Piext}_\infty$ takes values in the category $\LCCC_\infty$ of locally cartesian closed quasicategories and locally cartesian closed functors. As before, it takes equivalences of contextual categories to categorical equivalences of quasicategories.
\end{enumerate}
\end{theorem}

\begin{proof}
 The first part of (\ref{no:Cl-values-Lex}) is noted in \cite[p.\ 10]{kapulkin:locally-cartesian-qcat}. For the second part, observe that the equivalences of contextual categories are exactly the weak equivalences of their underlying fibration categories, which are in turn preserved by $\Nf$ \cite[Thm.\ 3.3]{szumilo:two-models}.
 
 Similarly, the first part of (\ref{no:Cl-values-LCCC}) is exactly the statement of \cite[Thm.\ 5.8]{kapulkin:locally-cartesian-qcat}, whereas the second part follows immediately by the same reasoning as above.
\end{proof}

In light of the above theorem one can formulate the following conjecture, an $\infty$-categorical analogue of the results of Clairambault and Dybjer \cite{clairmbault-dybjer:biequivalence}, establishing intensional type theory as an internal language for suitable $\infty$-categories.

\begin{conjecture} \label{conj:internal-languages}
 The functors 
 \begin{align*}
  \Cl^{\Id,\unit,\synSigma}_\infty & :  \CxlCat_{\Id,\unit,\synSigma} \to \Lex_\infty
  \\
  \Cl^{\Id,\unit,\synSigma, \Piext}_\infty & : \CxlCat_{\Id,\unit,\synSigma, \Piext} \to \LCCC_\infty
 \end{align*}
  are DK-equivalences of categories with weak equivalences.
\end{conjecture}

Ultimately, one would like to extend the above correspondences to include univalent type theories on one side and elementary $\infty$-toposes on the other; however, neither of these notions is yet defined.
On the type-theoretic side, it is not currently clear which rules to choose, of the many proposed for univalent universes and higher inductive types.
On the higher-categorical side, a precise definition of an elementary $\infty$-topos remains to be formulated.
Lurie \cite{lurie:htt} provides a detailed study of \emph{Grothendieck} $\infty$-toposes, but does not pursue the idea of their elementary counterparts \cite[6.1.3.11]{lurie:htt}.

Once these notions have been formulated, one hopes that the functor $\Cl^{\Id,\unit,\synSigma,\Piext}_\infty$ can be promoted to a functor $\Cl^\HoTT_\infty : \CxlCat_\HoTT \to \ElTop_\infty$ (where we write $\CxlCat_\HoTT$ for the category of contextual categories admitting rules for univalent type theories, and $\ElTop_\infty$ for the category of elementary $\infty$-toposes).
Eventually, one hopes to obtain the following diagram, with the horizontal arrows DK-equivalences:

  \[\begin{tikzpicture}[cd-style,yscale=1.2,xscale=5]
    \node (HoTT) at (0,2) {$\CxlCat_\HoTT$};
    \node (DTTPi) at (0,1) {$\CxlCat_{\Id,\unit,\synSigma, \Piext}$};
    \node (DTT) at (0,0) {$\CxlCat_{\Id,\unit,\synSigma}$};
    \node (ElemTop) at (1,2) {$\ElTop_\infty$};
    \node (LCCC) at (1,1) {$\LCCC_\infty$};
    \node (Lex) at (1,0) {$\Lex_\infty$};
    \draw[cd-arrow-style,dashed] (HoTT) to (DTTPi);
    \draw[cd-arrow-style] (DTTPi) to (DTT);
    \draw[cd-arrow-style,dashed] (ElemTop) to (LCCC);
    \draw[cd-arrow-style] (LCCC) to (Lex);
    \draw[cd-arrow-style,dashed] (HoTT) to node {$\Cl^\HoTT_\infty$} node[swap] {$(\sim)$} (ElemTop);
    \draw[cd-arrow-style] (DTT) to node {$\Cl^{\Id,\unit,\synSigma}_\infty$}  node[swap] {$(\sim)$}  (Lex);
    \draw[cd-arrow-style] (DTTPi) to node {$\Cl^{\Id,\unit,\synSigma, \Piext}_\infty$} node[swap] {$(\sim)$}  (LCCC);
  \end{tikzpicture}\]
  
\begin{remark}
  The categories of the right hand column may be seen as the $(\infty,1)$-cores of larger $(\infty,2)$-categories.
  One may wonder if the maps $\Cl_\infty$ are in fact $(\infty,2)$-equivalences, for some yet-to-be-defined $(\infty,2)$-category structures on $\CxlCat_{(\ldots)}$.

  By analogy with 1-categorical settings, we hope that this should be the case for $\Cl^{\Id,\unit,\synSigma}_\infty$, but do not expect it for $\Cl^{\Id,\unit,\synSigma, \Piext}_\infty$ or $\Cl^\HoTT_\infty$.
  Indeed, we do not expect the full $(\infty,2)$-category $\LCCC_\infty$ to be as well-behaved at all as its $(\infty,1)$-core, essentially due to the non-covariance of exponentials.
  
  This phenomenon appears most simply in, for example, the fact that if $\C[A]$ is the free cartesian closed category on an object, and $F, G : \C[A] \to \D$ are cartesian functors (determined by the objects $FA, GA \in \D$), then natural \emph{isomorphisms} $\alpha : F \iso G$ are determined uniquely by isomorphisms $\alpha_A : FA \iso GA$, but no such nice property holds for general natural transformations $F \to G$.
\end{remark}

Since this paper was first made publicly available in 2016, a part of Conjecture \ref{conj:internal-languages} was proven in \cite{kapulkin-szumilo:internal-language-lex}.
Specifically, the $\infty$-category $\Lex_\infty$ is shown there to be equivalent to the $\infty$-category of comprehension categories with $\Id$, $\unit$, $\synSigma$-types, which reduces the conjecture to a comparison between contextual categories and comprehension categories with the appropriate structure.

\subsection{Fibrations and cofibrations}

We define in this section two cofibrantly generated weak factorization systems $(\Cof,\TFib)$ and $(\Anod,\Fib)$ on $\CxlCat_{\Id,\unit,\synSigma (, \Piext)}$.
To do so, we first set up generating sets of left maps, whose domains and codomains are presented in Definition \ref{def:free-cxl-cats}, as freely generated objects in $\CxlCat_{\Id,\unit,\synSigma}$.
(The existence of freely generated objects follows from the presentation of $\CxlCat_{\Id,\unit,\synSigma(,\Piext)}$ as models of an essentially algebraic theory.)

\begin{definition} \label{def:free-cxl-cats}
 Define the following freely generated objects in $\CxlCat_{\Id,\unit,\synSigma}$:
\begin{enumerate}
  \item $\freeCxlCat{\Gamma_n}$ is freely generated by a context of length $n$.

  \item $\freeCxlCat{\Gamma_n \types A}$ is freely generated by a context $\Gamma$ of length $n$, and a type $A$ over this context.
(Of course $\freeCxlCat{\Gamma_n \types A} \iso \freeCxlCat{\Gamma_{n+1}}$, but we distinguish them notationally for readability.)

  \item $\freeCxlCat{\Gamma_n \types a : A}$ is freely generated by $\Gamma$, $A$ as in  $\freeCxlCat{\Gamma_n \types A}$, and a section of $p_A$.

  \item $\freeCxlCat{\Gamma_n \types A \equiv A'}$ is freely generated by a context $\Gamma$ of length $n$, types $A$, $A'$ over $\Gamma$, and maps $f, g_l, g_r, \alpha_l, \alpha_r$ constituting an equivalence from $A$ to $A'$ over $\Gamma$.

  \item $\freeCxlCat{\Gamma_n \types e : \Id_A(a,a')}$ is freely generated by $\Gamma$, $A$ as in  $\freeCxlCat{\Gamma_n \types A}$, and a section of the composite projection map $\Gamma.A.A.\Id_A \to \Gamma$ (giving all three: $a$, $a'$,  and $e$).
\end{enumerate}  
\end{definition}

Applying the left adjoint functor $F : \CxlCat_{\Id,\unit,\synSigma} \to \CxlCat_{\Id,\unit,\synSigma,\Piext}$ gives similarly freely generated objects in $\CxlCat_{\Id,\unit,\synSigma,\Piext}$.
When necessary for disambiguation, we may distinguish these different incarnations as e.g.\ $\freeCxlCat[\Id,\unit,\synSigma]{\Gamma_n}$ vs.\ $\freeCxlCat[\Id,\unit,\synSigma,\Piext]{\Gamma_n}$; but when it is clear which category we are working in, or when statements apply to both of them, we write just $\freeCxlCat{\Gamma_n}$, and so on.

\begin{definition} \label{def:generating-left-maps}
  Take $I$ and $J$ to be the following sets of maps in $\CxlCat_{\Id,\unit,\synSigma(,\Piext)}$:

\begin{enumerate}
  \item $I$ consists of the evident inclusions $\freeCxlCat{\Gamma_n} \to \freeCxlCat{\Gamma_n \types A}$ and $\freeCxlCat{\Gamma_n \types A} \to \freeCxlCat{\Gamma_n \types a:A}$, for all $n \in \N$;

  \item $J$ consists of the evident inclusions $\freeCxlCat{\Gamma_n \types A} \to \freeCxlCat{\Gamma_n \types A \equiv A'}$ and $\freeCxlCat{\Gamma_n \types a:A} \to \freeCxlCat{\Gamma_n \types e:\Id_A (a,a')}$, for all $n \in \N$.
\end{enumerate}
\end{definition}

\begin{definition}\label{def:cont-cat-classes}
  In each of $\CxlCat_{\Id,\unit,\synSigma}$ and $\CxlCat_{\Id,\unit,\synSigma, \Piext}$, we define the classes of maps $\TFib := \rorth{I}$, $\Cof := \lorth{\TFib}$, $\Fib := \rorth{J}$, and $\Anod := \lorth{\Fib}$.
  Call maps in these classes \defemph{trivial fibrations}, \defemph{cofibrations}, \defemph{fibrations}, and \defemph{anodyne maps}.
\end{definition}

Unwinding the universal properties of the maps in $I$, we see that a map $F : \C \to \D$ is a trivial fibration just if types and terms lift along it on the nose (we will call these properties \defemph{strict type lifting} and \defemph{strict term lifting}); that is, for any $\Gamma \in \C$ and $A \in \Ty(F \Gamma)$, there is some $\bar{A} \in \Ty(\Gamma)$ with $F(\bar{A}) = A$, and similarly for terms.
Note that since clearly strict type/term lifting implies the corresponding weak version, every trivial fibration is also a weak equivalence.
These conditions are a strong form of conservativity, considered in \cite[Def.~4.2.5]{lumsdaine:thesis} as \defemph{contractibility}; cf.~also \cite[Thm.~3.2.5]{hofmann:thesis}.

Similarly, the lifting properties of a fibrations can be seen explicitly as \defemph{equivalence-lifting} and \defemph{path-lifting} respectively.
Note that checking this equivalence-lifting criterion directly would be rather tedious in practice, since it involves lifting \emph{structured} equivalences.
In Corollary~\ref{cor:cheap-fibrations} below, we show that this is happily unnecessary: it is enough to lift unstructured equivalences.

Using the small object argument, we have:
\begin{proposition}
  $(\Cof,\TFib)$ and $(\Anod,\Fib)$ are both weak factorization systems.
\end{proposition}

  The forgetful functor $\CxlCat_{\Id,\unit,\synSigma,\Piext} \to \CxlCat_{\Id,\unit,\synSigma}$ preserves and reflects fibrations and trivial fibrations, while its left adjoint preserves cofibrations and anodyne maps, since $I^{\Id,\unit,\synSigma,\Piext}$, $J^{\Id,\unit,\synSigma,\Piext}$ were the images of $I^{\Id,\unit,\synSigma}$, $J^{\Id,\unit,\synSigma}$ under the left adjoint.
  Note however that the forgetful functor will \emph{not} generally preserve cofibrations or anodyne maps.
 
  It also follows automatically that a map of contextual categories is a cofibration (resp.~anodyne) just if it is a retract of a cell complex built from the basic maps in $I$ (resp.~$J$).
  We will not make formal use of this fact, but it is helpful for intuition: a typical cofibration is an $I$-cell complex, i.e.\ an extension of type theories obtained by repeatedly adjoining new types and terms (possibly infinitely many), but no new judgemental equalities.
  In particular, a typical cofibrant object is a type theory generated (over the constructors $\Id,\unit,\synSigma$ or $\Id,\unit,\synSigma,\Piext$ under consideration) just by algebraic type and term rules, with no extra definitional equalities.
  Similarly, a typical anodyne map is a $J$-cell complex, i.e.\ an extension built by repeatedly adjoining new terms and types along with equivalences or propositional equalities to pre-existing ones.

\begin{proposition} \label{prop:all-objects-fibrant}
 Every object in $\CxlCat_{\Id,\unit,\synSigma (,\Piext)}$ is fibrant.
\end{proposition}

\begin{proof}
 Immediate since the generating anodyne maps all have retractions, given by the identity equivalence and reflexivity term, respectively.
\end{proof}

  We close by showing that fibrations interact with equivalences as one would hope.

\begin{proposition} \label{prop:tfib-iff-fib-and-weq}
  A map in $\CxlCat_{\Id,\unit,\synSigma(,\Piext)}$ is a trivial fibration if and only if it is both a fibration and a weak equivalence.
\end{proposition}

\begin{proof}
  It is clear that any trivial fibration is both a weak equivalence and a fibration.

  For the converse, suppose $F : \C \to \D$ is a weak equivalence and a fibration.
  
  Given a context $\Gamma$ in $\C$ and type $A$ over $F\Gamma$ in $\D$, we may find (since $F$ is a weak equivalence) some type $A'$ over $\Gamma$ in $\C$, together with an equivalence $w : F(A') \equiv A$ over $\Gamma$ in $\D$.
  Choose left and right quasi-inverses for $w$.
  Since $F$ is a fibration, we may now lift $A$ and $w$ together to $\C$.
  In particular, we have succeeded in lifting $A$ on the nose, as required.

  Strict lifting of terms is entirely analogous: first lift the term up to equivavalence (since $F \in \WEq$), and then use that equivalence to lift the original term on the nose (by $F \in \Fib$).   
\end{proof}


\section{Categories with attributes}

For assembling our classes of maps into semi-model structures on $\CxlCat_{\Id,\unit,\synSigma(,\Piext)}$, our main technical workhorse will be the category $\Eqv[\C]$ of \defemph{span-equivalences} in $\C$---almost a path object, but not quite---along with some related auxiliary constructions.

All these are most naturally viewed not directly as constructions on contextual categories, but as living in the slightly more general world of \defemph{categories with attributes} (CwA’s).

In this section, we therefore recall and develop some background results on CwA’s and their relationship with contextual categories, before tackling the span-equivalence constructions themselves in Section~\ref{sec:spans}.

\subsection{Categories with Attributes: background}

\begin{definition}
 A \defemph{category with attributes} (CwA) consists of:
 \begin{enumerate}
  \item a category $\C$, with a chosen terminal object $1$;
  \item a functor $\Ty : \C^\op \to \Set$;
  \item an assignment to each $A \in \Ty(\Gamma)$, an object $\Gamma.A \in \C$ and a map $p_A : \Gamma.A \to \Gamma$;
  \item for each $A \in \Ty(\Gamma)$ and $f : \Delta \to \Gamma$, a map $f.A : \Delta.f^*A \to \Gamma.A$ (called the \defemph{connecting map}) such that the following square is a pullback:
  \[\begin{tikzcd}
    \Delta.f^*A \ar[r, "f.A"] \ar[d, fib, "p_{f^*A}"'] \arrow[dr, phantom, "\lrcorner", very near start] & \Gamma.A \ar[d, fib, "p_A"] \\
    \Delta \ar[r, "f"] & \Gamma
  \end{tikzcd}\]  
 \end{enumerate}
\end{definition}

As defined, categories with attributes are models for an evident essentially algebraic theory.
A map of categories with attributes is a homomorphism of such models: explicitly, a functor $F : \C \to \C'$ and transformation $F_\Ty : \Ty_\C \to \Ty_{\C'}\cdot F$, strictly preserving all the structure (chosen terminal object, context extension, and so on).

Write $\CwA$ for the category of categories with attributes.
Just as in the case of contextual categories, one may equip categories with attributes with additional structure corresponding to different type constructors.
The translations of these structures from the language of contextual categories to that of categories with attributes are straightforward.
We will write $\CwA_{\Id, \unit, \synSigma}$ and $\CwA_{\Id, \unit, \synSigma, \Piext}$ for the categories of categories with attributes equipped with the corresponding extra structure, and when a statement applies to both of these cases, we will indicate it by writing $\CwA_{\Id, \unit, \synSigma(, \Piext)}$

\begin{definition}
  The presheaf $\Ty$ defined in Notation~\ref{notation:cxl-cats} allows us to regard any contextual category as a category with attributes.
  This extends to an evident faithful functor $\CxlCat \to \CwA$; and indeed exhibits $\CxlCat$ as the full subcategory consisting of CwA’s equipped with a suitable grading on objects, since such a grading is unique if it exists, and is automatically preserved by any CwA map.
\end{definition}

To go the other way, we generalize Definition~\ref{def:fibrant-slice-of-cxl-cat}:

\begin{definition} \label{def:fibrant-slice-of-cwa}
  Let $\C$ be a CwA, and $\Gamma$ an object of $\C$.
  
  A \defemph{context over  $\Gamma$} is a sequence $\Delta = (A_0,\cdots,A_{n-1})$, where $A_0 \in \Ty_\C(\Gamma)$, \ldots, $A_i \in \Ty_\C(\Gamma . A_0 . \cdots . A_{i-1})$, \ldots.
  Any context induces an evident \defemph{context extension} $\Gamma.\Delta := \Gamma . A_0 . \cdots . A_{n-1}$, with projection map $p_\Delta : \Gamma.\Delta \to \Gamma$.

  The \defemph{fibrant slice} $\fibslice{\C}{\Gamma}$ is the contextual category defined as follows:

  \begin{enumerate}
  \item objects of degree $n$ are contexts $\Delta$ of length $n$ over $\Gamma$;
  \item $\fibslice{\C}{\Gamma} (\Delta,\Delta') := \C/\Gamma(\Gamma . \Delta, \Gamma.\Delta')$, and the category structure is inherited from $\C/\Gamma$;
  \item reindexing and the connecting maps are inherited directly from $\C$.
  \end{enumerate}

  Moreover, an $\Id$-type (resp.\ $\unit$, $\synSigma$, extensional $\synPi$-type) structure on $\C$ induces one on $\fibslice{\C}{\Gamma}$.

  A map $f : \Gamma' \to \Gamma$ in $\C$ induces a contextual functor $f^* : \fibslice{\C}{\Gamma} \to \fibslice{\C}{\Gamma'}$, functorially in $f$, and preserving all logical structure under consideration.

  Similarly, for any CwA map $F : \C \to \D$ and object $\Gamma \in \C$, there is as before an induced slice functor $\fibslicefunc{F}{\Gamma} : \fibslice{\C}{\Gamma} \to \fibslice{\D}{F \Gamma}$, preserving any logical structure that $F$ does.
\end{definition}

  In particular, we call $\fibslice{\C}{1}$ the \defemph{contextual core} of $\C$, and denote this by $\core \C$.

\begin{proposition}
  The assignation $\C \mapsto \core \C$ extends to a functor $\core : \CwA \to \CxlCat$, right adjoint to the inclusion functor (so exhibiting $\CxlCat$ as a coreflective subcategory of $\CwA$), and similarly for the categories with ${\Id,\unit,\synSigma}$ and ${\Id,\unit,\synSigma,\Piext}$:

  \[\begin{tikzcd}[row sep = scriptsize]
    \CxlCat_{\Id,\unit,\synSigma,\Piext} \ar[rr, hook, bend left=10] \ar[d] \ar[rr, phantom, "\scriptstyle \bot"] & &
    \CwA_{\Id,\unit,\synSigma,\Piext} \ar[ll, bend left=10] \ar[d] \\ 
    \CxlCat_{\Id,\unit,\synSigma} \ar[rr, hook, bend left=10] \ar[d]  \ar[rr, phantom, "\scriptstyle \bot"] & &
    \CwA_{\Id,\unit,\synSigma} \ar[ll, bend left=10] \ar[d] \\ 
    \CxlCat \ar[rr, hook, bend left=10] \ar[rr, phantom, "\scriptstyle \bot"] & &
    \CwA \ar[ll, bend left=10] \\ 
  \end{tikzcd}\]
\end{proposition}

\begin{proof}
  Given $F : \C \to \D$, with $\C$ a contextual category and $\D$ an arbitrary CwA, $F$ factors uniquely through $\fibslice{\D}{1} \to \D$ via a contextual map $\bar{F} : \C \to \fibslice{\D}{1}$, since every $\Gamma \in \ob_n \C$ is uniquely expressible in the form $1.A_0 \ldots A_{n-1}$, and hence must be sent under $\bar{F}$ to the sequence $(FA_0,\ldots,FA_n)$.
  
  It is similarly routine to check that $\fibslice{\D}{1} \to \D$ preserves all logical structure under consideration, and given $F$ as above, $\bar{F}$ preserve such structure if and only if $F$ does. 
\end{proof}

\subsection{Equivalences in CwAs}

In a contextual category with identity types, Definition~\ref{def:equiv-in-cxl-cat} gives a good notion of when a map is an equivalence.

In a category with attributes, however, objects are not in general built up out of types (i.e., there may objects $\Gamma \in \C$ whose canonical map $\Gamma \to 1$ cannot be written as a composite of $p$-maps).
So we do not have identity contexts for arbitrary objects, nor hence a notion of homotopy between arbitrary maps; so we need a slightly less direct definition of equivalences.

\begin{definition} \label{def:equiv-in-cwa}
  Let $\C$ be a CwA with $\Id$-types.
  A map $\Gamma \to \Delta$ in $\C$ an \defemph{equivalence} if the induced contextual functor $f^* : \fibslice{\C}{\Delta} \to \fibslice{\C}{\Gamma}$ is an equivalence of contextual categories (in the sense of Definition~\ref{def:equiv-of-cxl-cats}).
\end{definition}

When $\C$ is contextual, or more generally when $f$ lies in some fibrant slice of $\C$, this coincides with the established definition:

\begin{proposition} \label{prop:map-equiv-iff-pullback-equiv}
  Let $\C$ be a CwA with $\Id$-types, and $f : \Gamma.\Delta_1 \to \Gamma.\Delta_2$ a map between context extensions over $\Gamma \in \C$.
  Then $f$ is an equivalence in the sense of Definition~\ref{def:equiv-in-cwa} if and only if, considered as a map $\Delta_1 \to \Delta_2$ in the contextual category $\fibslice{\C}{\Gamma}$, it is an equivalence in the sense of Definition~\ref{def:equiv-in-cxl-cat}.
\end{proposition}

\begin{proof}
  $(\Imp)$:
  Suppose $f$ is an equivalence in the sense of Definition~\ref{def:equiv-in-cwa}; that is, $f^* : \fibslice{\C}{\Gamma.\Delta_2} \to \fibslice{\C}{\Gamma.\Delta_1}$ is an equivalence of contextual categories.
Note that we have $f^*(\Gamma.\Delta_2.p_{\Delta_2}^* \Delta_1) = \Gamma.\Delta_1.p_{\Delta_1}^* \Delta_1$, since the lower composite in the diagram
  \[\begin{tikzcd}
    \Gamma.\Delta_1.p_{\Delta_1}^*\Delta_1 \ar[r] \ar[d] 
       & \Gamma.\Delta_2.p_{\Delta_2}^*\Delta_1 \ar[d] \ar[r]
       & \Gamma.\Delta_1 \ar[d] \\
    \Gamma.\Delta_1 \ar[r, "f"] 
       & \Gamma.\Delta_2 \ar[r]
       & \Gamma
  \end{tikzcd}\] 
  is precisely $p_{\Delta_1}$.
  Thus, by Lemma \ref{lem:cxt-sec-lifting}, we can lift the canonical element $\Gamma,\, y_1 \of \Delta_1 \types y_1 : \Delta_1$ along $f^*$  to get $\Gamma,\, y_2 \of \Delta_2 \types g(y_2) : \Delta_1$ with $\Gamma,\, y_1 \of \Delta_1 \types \eta(y_1) : \Id_{\Delta_1}(gf(y_1), y_1)$.
  This defines $g \colon \Delta_2 \to \Delta_1$ in $\fibslice{\C}{\Gamma}$ and shows that it is a left quasi-inverse of $f$. 

  We claim that $g$ is also a right quasi-inverse of $f$.
  For that, consider the identity context $\Gamma,\, y_2 \of \Delta_2 \types \Id_{\Delta_2} (fg(y_2),y_2)$.
  Pulled back along $f$, it gives $\Gamma,\, y_1 \of \Delta_1 \types \Id_{\Delta_1} (fgf(y_1), f(y_1))$, which is inhabited by the action of $f$ on $\eta(y_1)$.
  Lifting this along $f^*$ gives $\Gamma,\, y_2 \of \Delta_2 \types \varepsilon(y_2): \Id_{\Delta_2} (fg(y_2),y_2)$ and hence taking $g_1 = g_2 = g$ along with $\eta$ and $\varepsilon$ above shows that $f$ is an equivalence in the sense of Definition~\ref{def:equiv-in-cxl-cat}.

  $(\Pmi)$:
  Assuming that $f$ is an equivalence in $\fibslice{\C}{\Gamma}$ in the sense of Definition~\ref{def:equiv-in-cxl-cat}, we may choose a (two-sided) quasi-inverse $g$ for it.
    
  For weak type lifting, suppose we have $\Gamma.\Delta_2.\Phi$ in $\fibslice{\C}{\Gamma.\Delta_2}$ along with $\Gamma.\Delta_2.f^*\Phi.A$ in $\fibslice{\C}{\Gamma.\Delta_1}$. We first pull the latter context back along $g$, forming $\Gamma.\Delta_2.g^*f^*\Phi.g^*A$. Transporting that along the homotopy $fg \homot \id$, we obtain $\Gamma.\Delta_2.\Phi.\overline{A}$. Finally, to show that $f^*\overline{A} \simeq A$ over $f^*\Phi$, we use the other homotopy $gf \homot \id$.
  
  The proof of weak term lifting is analogous. Given $\Gamma.\Delta_2.\Phi.A$ in $\fibslice{\C}{\Gamma.\Delta_2}$ and a section $a : \Gamma.\Delta_1.f^*\Phi \to \Gamma.\Delta_1.f^*\Phi.f^*A$, we pull $a$ back along $g$ and then correct it using $fg \homot \id$ to obtain a section $\overline{a} : \Gamma.\Delta_2.\Phi \to \Gamma.\Delta_2.\Phi.A$. The homotopy between $f^* \overline{a}$ and $a$ is then constructed using $gf \homot \id$.
\end{proof}

(Note that in the argument for $(\Imp)$, we did not use the weak type lifting property along $f^*$, only the weak term lifting; so as a scholium we see that for functors of the form $f^* : \fibslice{\C}{\Gamma.\Delta} \to \fibslice{\C}{\Gamma.\Delta'}$, where $f$ is a map over $\Gamma$, weak term lifting implies weak type lifting.)

\begin{proposition} \label{prop:2-of-6-in-cwas}
  Equivalences in a CwA satisfy 2-out-of-6, and are stable under retracts.
\end{proposition}

\begin{proof}
  By the same properties for equivalences of contextual categories (Corollary~\ref{cor:2-of-6-cxl-cats}).
\end{proof}

\begin{proposition} \label{prop:equiv-stable-under-subst-in-cwas}
  Suppose $\Delta_1, \Delta_2$ are context extensions of $\Gamma$, $w : \Gamma.\Delta_1 \to \Gamma.\Delta_2$ is a map over $\Gamma$, and $f : \Gamma' \to \Gamma$ is any map. If $w$ is an equivalence (in $\C$), then so is $f^*w : f^*\Delta_1 \to f^*\Delta_2$.
\end{proposition}

\begin{proof}
  We can view $w$ as a map in $\fibslice{\C}{\Gamma}$; by Proposition~\ref{prop:map-equiv-iff-pullback-equiv}, it is an equivalence there in the contextual sense.
  But $f^* : \fibslice{\C}{\Gamma} \to \fibslice{\C}{\Gamma'}$ is a contextual functor preserving $\Id$-types; so $f^*w$ is an equivalence in $\fibslice{\C}{\Gamma'}$, and hence in $\C$.
\end{proof}

\begin{proposition} \label{prop:right-properness-in-cwas}
  If $f : \Gamma' \to \Gamma$ is an equivalence, and $A \in \Ty(\Gamma)$, then $f.A : \Gamma'.f^*A \to \Gamma.A$ (the pullback of $f$ along $p_A$) is again an equivalence.
\end{proposition}

\begin{proof}
  $(f.A)^*$ is equal (on the nose!) to the functor $\fibslicefunc{f^*}{p_A}$, which is an equivalence by Lemma \ref{lem:fibslicefunc-equiv}.
\end{proof}

This can be seen as a form of right properness for equivalences in a CwA.

\subsection{Fibrations and cofibrations of CwAs}

In this section, we define classes of fibrations between CwAs analogously to how they were defined for contextual categories in Section \ref{sec:classes-of-maps}.
As before, we start with the generating sets of left maps in $\CwA_{\Id,\unit,\synSigma(,\Piext)}$, generalizing Definition~\ref{def:free-cxl-cats}.

\begin{definition} \label{def:free-cwas} \leavevmode
\begin{enumerate}
 \item $\freeCwA{\Gamma}$ is freely generated (as a CwA with $\Id$, $\unit$, $\synSigma$, and possibly $\Piext$) by a single object $\Gamma \in \C$.
 
 \item $\freeCwA{\Gamma \types A}$ is freely generated by $\Gamma \in \C$ and $A \in \Ty(\Gamma)$.

  \item $\freeCwA{\Gamma \types a : A}$ is freely generated by $\Gamma$, $A$ as above, and a section $a$ of $p_A : \Gamma.A \to \Gamma$.

  \item $\freeCwA{\Gamma \types A \equiv A'}$ is freely generated by a context $\Gamma$, types $A, A' \in \Ty(\Gamma)$, and maps $f, g_l, g_r, \alpha_l, \alpha_r$ constituting an equivalence $\Gamma.A \equiv \Gamma.A'$ in $\fibslice{\C}{\Gamma}$.

  \item $\freeCwA{\Gamma \types e : \Id_A(a,a')}$ is freely generated by $\Gamma$, $A$ as above, and a section of the iterated projection $\Gamma.A.A.\Id_A \to \Gamma$.
\end{enumerate}
Again, we disambiguate as e.g.\ $\freeCwA[{\Id,\unit,\synSigma}]{\Gamma \types A}$ when necessary; but it is never necessary.
\end{definition}

\begin{definition}\label{def:generating-left-maps-in-cwas}
  Take $I$ and $J$ to be the following sets of maps in $\CwA_{\Id,\unit,\synSigma(,\Piext)}$:
\begin{enumerate}
 \item $I$ consists of the inclusions $\freeCwA{\Gamma} \to \freeCwA{\Gamma \types A}$ and $\freeCwA{\Gamma \types A} \to \freeCwA{\Gamma \types a : A}$;
 \item $J$ consists of the inclusions $\freeCwA{\Gamma \types A} \to \freeCwA{\Gamma \types A \equiv A'}$ and $\freeCwA{\Gamma \types a : A} \to \freeCwA{\Gamma \types p : \Id_A(a,a')}$.
\end{enumerate}
\end{definition} 

Thus the sets $I$ and $J$ contain just two maps each, in contrast with the generating left maps for $\CxlCat_{\Id,\unit,\synSigma(,\Piext)}$, where we required infinitely many maps due to the grading of objects. 

\begin{definition}
  A map $F : \C \to \D$ in $\CwA_{\Id,\unit,\synSigma}$ (resp.\ $\CwA_{\Id,\unit,\synSigma,\Piext}$) is a \defemph{local fibration} (resp.\ \defemph{local trivial fibration}) if it is right-orthogonal to the maps $I$ (resp.~$J$) of Definition~\ref{def:generating-left-maps-in-cwas}.
\end{definition}

Just as in the case of contextual categories (Definition \ref{def:cont-cat-classes}), one may unwind this orthogonality to describe local (trivial) fibrations explicitly in terms of type/term lifting.
Specifically, a map $F : \C \to \D$ of CwA's is a local fibration exactly when
\begin{enumerate}
 \item given any $\Gamma \in \C$, $A \in \Ty_\C\Gamma$, $B \in \Ty_\D (F\Gamma)$, and structured equivalence $w : FA \equiv B$ over $F\Gamma$, there exists a lift $\bar{B} \in \Ty_\C \Gamma$ together with a structured equivalence $\bar{w} : A \equiv \bar{B}$ over $\Gamma$ such that $F \bar{w} = w$;
 \item given any $\Gamma \in \C$, $A \in \Ty_\C\Gamma$, a section $a$ of the projection $p_A$ in $\C$, a section $a'$ of $p_{FA}$ in $\D$, and a section $e$ of $p_{\Id_{FA}(Fa,a')}$, there exist lifts of $a'$, $e$ to $\C$.
\end{enumerate}
and a local trivial fibration just when types and terms lift along it on the nose.

Several useful facts follow immediately from this description:

\begin{proposition} \label{prop:classes-and-core} \leavevmode \
  \begin{enumerate}
  \item A contextual functor is a (trivial) fibration in the sense of Definition~\ref{def:cont-cat-classes} exactly if, viewed as a map of CwA’s, it is a local (trivial) fibration.
  \item A map $F : \C \to \D$ of CwA’s is a local (trivial) fibration exactly if all its slice functors $\fibslicefunc{F}{\Gamma} : \fibslice{\C}{\Gamma} \to \fibslice{\D}{F\Gamma}$ are (trivial) fibrations. \label{item:local-fibs-are-local}
  \item The functors $\core : \CwA_{\Id,\unit,\synSigma(,\Piext)} \to \CxlCat_{\Id,\unit,\synSigma(,\Piext)}$ send local (trivial) fibrations to (trivial) fibrations.
  \item The inclusion functors $\CxlCat_{\Id,\unit,\synSigma(,\Piext)} \to \CwA_{\Id,\unit,\synSigma(,\Piext)}$ perserve the corresponding left classes (by adjunction from the previous statement).  \thmqed
  \end{enumerate}
\end{proposition}

\begin{remark} \label{rmk:no-local-2-of-3}
 Note that we use the word \emph{local} here in the sense of a property defined slice-wise, rather than in its more common homotopy-theoretic sense of a property defined homset-wise.
 We avoid calling them just “(trivial) fibrations” since they do not behave the way one would expect such classes to behave, as their lifting properties are only for terms and types, not for arbitrary objects.
 In particular, local trivial fibrations do not satisfy “relative 2-out-of-3” among local fibrations.
 To see this, consider some map $F : \C \to \C'$, and another CwA $\D$.
 Then the inclusions $\D \to \C + \D$ and $\D \to \C' + \D$ are local trivial fibrations, and $F + \id_\D : \C + \D \to \C' + \D$ is a local fibration; but $F + \id_\D$ is a local trivial fibration only if $F$ was.
\end{remark}

As mentioned at the end of Section \ref{sec:classes-of-maps}, the definition of a (local) fibration is slightly tedious to check directly, as it involves lifting structured equivalences.
Happily, this can be simplified:

\begin{lemma} \label{lem:cheap-fibrations-aux}
  Let $F : \C \to \D$ be a map of CwA’s with $\Id$-types, satisfying the path-lifting property for local fibrations, i.e.\ orthogonal to $\freeCwA{\Gamma \types a : A} \to \freeCwA{\Gamma \types p : \Id_A(a,a')}$.
  Then the following are equivalent:
  \begin{enumerate}[(a)]
   \item \label{item:struc}
     $f$ satisfies the equivalence-lifting property of a local fibration, i.e.~ is orthogonal to $\freeCwA{\Gamma \types A} \to \freeCwA{\Gamma \types A \equiv A'}$;
   \item \label{item:unstruc-out}
     for any $A \in \Ty_\C(\Gamma)$, $A' \in \Ty_\D(F\Gamma)$, and (unstructured) equivalence $f : F \Gamma. F A \weqto F \Gamma . A'$ over $\Gamma$, there are lifts $\bar{A'} \in \Ty_\C(\Gamma)$ of $A'$ and $\bar{f} : \Gamma.A \weqto \Gamma.\bar{A'}$ of $f$;
   \item \label{item:unstruc-in}
     for any $A \in \Ty_\C(\Gamma)$, $A' \in \Ty_\D(F\Gamma)$, and (unstructured) equivalence $g :  F \Gamma . A' \weqto F \Gamma. F A$ over $\Gamma$, there are lifts $\bar{A'} \in \Ty_\C(\Gamma)$ of $A'$ and $\bar{g} : \Gamma.\bar{A'} \weqto \Gamma.A$ of $g$.
  \end{enumerate} 
\end{lemma}

\begin{proof}
  (\ref{item:struc}) $\Imp$ (\ref{item:unstruc-out}), (\ref{item:unstruc-in}) is immediate.
  
  (\ref{item:unstruc-out}) $\Imp$ (\ref{item:struc}): given types $A \in \Ty_\C(\Gamma)$, $A' \in \Ty_\D(F\Gamma)$, and a structured equivalence $(f,g_1,\eta,g_2,\varepsilon)$ from $F A$ to $A'$ over $F \Gamma$, we need to lift the whole structured equivalence, on the nose.
  By (\ref{item:unstruc-out}), $f$ lifts to an equivalence $\bar{f} : A \to \bar{A'}$, for which we may choose weak inverse data $(g_l',\eta',g_2',\varepsilon')$.
  Then $(Fg_l',F\eta',Fg_2',F\varepsilon')$ give alternate weak inverse data for $f$.
  By essential uniqueness of such data, $F g_l'$ is propositionally equal to $g_1$, so by the path-lifting property, we can lift $g_1$ (and the connecting equality) on the nose, and similarly for $\eta$, $g_2$, $\varepsilon$ in turn.

  (\ref{item:unstruc-in}) $\Imp$ (\ref{item:unstruc-out}): suppose $g :  F \Gamma . A' \weqto F \Gamma. F A$ is as in (\ref{item:unstruc-out}).
  Choose some weak inverse $f : F \Gamma . F A \to F \Gamma . A'$ for $g$ over $\Gamma$.
  By (\ref{item:unstruc-in}), lift $A'$, $f$ to some $\bar{A'} \in \Ty_\C(\Gamma)$, $\bar{f} : \Gamma.A \weqto \Gamma.\bar{A'}$.
  Choose some weak inverse $g'$ for $\bar{f}$.
  Now $g$ and $F g'$ are both weak inverses for $f$, so are propositionally equal; so by the path-lifting property, we can lift $g$, as desired.
\end{proof}

\begin{corollary} \label{cor:cheap-fibrations}
  A map $F : \C \to \D$ of CwA’s with $\Id$-types is a local fibration (so, if $\C$, $\D$ are contextual, a fibration) if and only if it satisfies the path-lifting property from the definition, together with any one of the equivalent equivalence-lifting properties of Lemma~\ref{lem:cheap-fibrations-aux}. \thmqed
\end{corollary}

Finally, we generalise equivalences as well to a local analogue for CwA’s.

\begin{definition}  \label{def:local-eqiuvalence} 
  Say a map $F : \C \to \D$ of CwA’s with $\Id$-types is a \defemph{local equivalence} just if all its slice functors $\fibslicefunc{F}{\Gamma} : \fibslice{\C}{\Gamma} \to \fibslice{\D}{F\Gamma}$ are equivalences of contextual categories.
\end{definition}

\begin{proposition}
  Let $\C \to[F] \D \to[G] \E$ be maps of CwA’s with at least $\Id$-types.
  If $F$ and $G$ are local equivalences, then so is $GF$.
  If $G$ and $GF$ are local equivalences, then so is $F$. 
\end{proposition}

\begin{proof}
  Immediate by 2-out-of-3 for equivalences of contextual categories (Corollary~\ref{cor:2-of-6-cxl-cats}), together with the fact that $\fibslicefunc{GF}{\Gamma} = (\fibslicefunc{G}{F\Gamma})(\fibslicefunc{F}{\Gamma})$.
\end{proof}

\begin{remark}
  Note however that local equivalences do \emph{not} satisfy the remaining part of the 2-out-of-3 property, as shown again by Remark~\ref{rmk:no-local-2-of-3}.
\end{remark}

\begin{proposition}
  A map in $\CwA_{\Id,\unit,\synSigma(,\Piext)}$ is a local trivial fibration if and only if it is both a local fibration and a local equivalence.
\end{proposition}

\begin{proof}
  Immediate by Propositions~\ref{prop:classes-and-core}(\ref{item:local-fibs-are-local}) and \ref{prop:tfib-iff-fib-and-weq}.
\end{proof}

 
\section{The Reedy span-equivalences construction}
\label{sec:spans}

In this section, from a given CwA $\C$ with $\Id$-types, we describe four new CwA’s:

\begin{enumerate}
\item $\Eqv[\C]$, the CwA of \defemph{span-equivalences} in $\C$;
\item $\EqvRefl[\C]$, the CwA of \defemph{trivial auto-(span\nobreakdash\nobreakdash-)equivalences} in $\C$;
\item $\EqvComp[\C]$, the CwA of \defemph{homotopy-commutative triangles of (span\nobreakdash-)equi\-va\-len\-ces} in $\C$.
\item $\EqvInv[\C]$, the CwA of \defemph{mutually inverse pairs of (span\nobreakdash-)equi\-va\-len\-ces} in $\C$.
\end{enumerate}

Each of these is constructed as the CwA of homotopical diagrams in $\C$ on a suitable homotopical inverse category.
Recall that a \defemph{homotopical category} is a category with weak equivalences, satisfying the 2-out-of-6 property.
A homotopical diagram in a CwA $\C$ is therefore a functor from a small homotopical category $(\I, \WEq)$ to $\C$ taking $\WEq$ to the equivalences in $\C$ in the sense of the Definition \ref{def:equiv-in-cwa}.

The general construction of CwA’s of homotopical diagrams on inverse categories, and logical structure on them, will be given in the companion paper \cite{kapulkin-lumsdaine:inverse-diagrams}.
The types in these CwA’s are analogous to \defemph{Reedy fibrations} of diagrams in a fibration category; their construction is thus in large part translating constructions of \cite{shulman:inverse-diagrams} from the language of fibration categories to the language of CwA’s (and more generally comprehension categories).

For each of our four constructions on CwA’s, we therefore set up the appropriate ordered homotopical inverse category on which to take diagrams; give an explicit description of the resulting CwA; and note a few facts about the result.

Precisely, the notions from \cite{kapulkin-lumsdaine:inverse-diagrams} we require are the following:

\begin{definition}[\cite{kapulkin-lumsdaine:inverse-diagrams}]
  An \defemph{ordering} on an inverse category $\I$ consists of, for each $i \in \I$, a finite total ordering $<_i$ on arrows out of $i$ such that for any $i \to[f] j \to[g] k$, we have $gf <_i f$.
\end{definition}

(Orderings will be used for constructing matching objects as context extensions; cf.\ Remark~\ref{rmk:orderings-for-cxts} below.)

\begin{proposition}[{\cite[Thm.~7.1]{kapulkin-lumsdaine:inverse-diagrams}}] \label{prop:deferred-results} \leavevmode
  \begin{enumerate}
  \item For any CwA $\C$ with $\Id$-types, and any ordered homotopical inverse category $(\I,\WEq)$, there is a CwA $\C^\I_h$, whose objects are homotopical $\I$-diagrams in $\C$, and whose types are “homotopical Reedy $\I$-types” in $\C$.

    (The general definition of $\C^\I_h$ is somewhat involved to state; for the cases we use, we will recall the resulting CwA explicitly.)
  
  \item $\C^\I$ carries $\Id$-types; and if $\C$ carries $\unit$- and $\synSigma$-types, so does $\C^\I$.

  \item If $\C$ carries $\Piext$-types, and additionally all maps of $\I$ are equivalences, then $\C^\I$ carries $\Piext$-types.

  \item A CwA map $F : \C \to \D$ induces a CwA map $F^\I : \C^\I \to \D^\I$, preserving whatever logical structure $F$ preserved, and functorially in $F$.

  \item Any order-preserving homotopical discrete opfibration $f : \I \to \J$ induces a map $\C^f : \C^\J \to \C^\I$, preserving all logical structure, and functorially in $f$. \label{item:contravariance}

  \item If $f :\I \to \J$ as above is moreover injective, then $\C^f$ is a local fibration; and if $f$ is a homotopy equivalence, then $\C^f$ is a local equivalence.
  \end{enumerate}
\end{proposition}

\begin{remark}
  For the individual instances we require, the proofs of the above facts are all straightforward verifications, albeit rather lengthy and containing much shared material.
  As such, we originally planned to give them individually in the present paper, before realizing they were sufficiently repetitive that it was better to develop the construction in generality.
\end{remark}

For the whole of this section, fix some CwA $\C$ with $\Id$-types.

\subsection{The CwA of Reedy Spans}

\begin{definition}
  $\SpanCat$ is the inverse category $(0) \from (01) \to (1)$ (which may be trivially considered as a homotopical category with no maps marked as equivalences).
  $\EqvCat$ is the homotopical category on $\SpanCat$, with all maps marked as equivalences.
  
  We order these by taking $(0) <_{(01)} (1)$.
  (Here and below, when coslices are posetal, we present the orderings concisely by identifying objects of coslices with their codomains; moreover, we spell out the ordering only on the \emph{proper} part of each coslice, since $\id_i$ must always be the top element of $<_i$.)
\end{definition}

\begin{definition} \label{def:cwa-of-span-equivs}
  $\Span[\C]$ (resp.\ $\Eqv[\C]$) is the CwA of (homotopical) diagrams on $\SpanCat$ (resp.\ $\EqvCat$) in $\C$.
  
  Concretely, $\Span[\C]$ can be described as follows:
  \begin{enumerate}
  \item objects $\GGamma$ are spans $\Gamma_0 \from[l_0] \Gamma_{01} \to[l_1] \Gamma_1$ in $\C$;
  \item maps $\ff : \DDelta \to \GGamma$ are natural transformations between spans
    
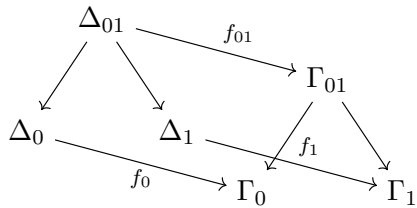
\[\begin{tikzpicture}[cd-style,x={(1cm,0cm)},y={(0cm,1.5cm)},z={(-3cm,0.8cm)}]
      \node (G0) at (0,0,0) {$\Gamma_0$};
      \node (G1) at (2,0,0) {$\Gamma_1$};
      \node (G*) at (1,1,0) {$\Gamma_{01}$};
      \draw[cd-arrow-style] (G*) -- (G0);
      \draw[cd-arrow-style] (G*) -- (G1);
      \node (D0) at (0,0,1) {$\Delta_0$};
      \node (D1) at (2,0,1) {$\Delta_1$};
      \node (D*) at (1,1,1) {$\Delta_{01}$};
      \draw[cd-arrow-style] (D*) -- (D0);
      \draw[cd-arrow-style] (D*) -- (D1);
      \draw[cd-arrow-style] (D0) -- node[below] {$f_0$} (G0);
      \draw[cd-arrow-style] (D1) -- node {$f_1$} (G1);
      \draw[cd-arrow-style] (D*) -- node {$f_{01}$} (G*);
    \end{tikzpicture}\]
  \item types over an object $\GGamma = (\Gamma_0 \from[l_0] \Gamma_{01} \to[l_1] \Gamma_1)$ are triples $\AA = (A_0,A_1,A_{01})$, where $A_0 \in \Ty(\Gamma_0)$, $A_1 \in \Ty(\Gamma_1)$, and $A_{01} \in \Ty(\Gamma_{01}.l_0^*A_0.\pi_{l_0^*A_0}^*l_1^*A_1)$, with the context extension $(\Gamma_0 \from[l_0] \Gamma_{01} \to[l_1] \Gamma_1).(A_0,A_1,A_{01})$ and projection map as given by the following diagram:
    
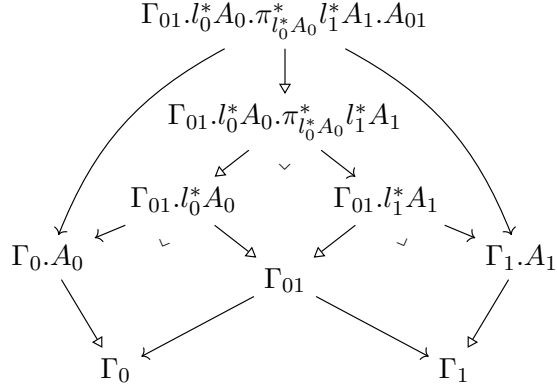
\[\begin{tikzpicture}[cd-style,x={(1.5cm,0cm)},y={(0cm,1.5cm)}]
      \node (G0) at (-1.5,0) {$\Gamma_0$};
      \node (G1) at (1.5,0) {$\Gamma_1$};
      \node (G*) at (0,0.8) {$\Gamma_{01}$};
      \draw[cd-arrow-style] (G*) -- (G0);
      \draw[cd-arrow-style] (G*) -- (G1);
      \node (GA0) at (-2.1,1) {$\Gamma_0.A_0$};
      \draw[cd-arrow-style,fib] (GA0) -- (G0);
      \node (GlA0) at (-0.9,1.5) {$\Gamma_{01}.l_0^*A_0$};
      \draw[cd-arrow-style,fib] (GlA0) -- (G*);
      \draw[cd-arrow-style] (GlA0) -- (GA0);
      \draw[cd-arrow-style,description,phantom] (GlA0) -- node[sloped,very near start] {\rotatebox{45}{$\ulcorner$}} (G0);
      \node (GA1) at (2.1,1) {$\Gamma_1.A_1$};
      \draw[cd-arrow-style,fib] (GA1) -- (G1);
      \node (GlA1) at (0.9,1.5) {$\Gamma_{01}.l_1^*A_1$};
      \draw[cd-arrow-style,fib] (GlA1) -- (G*);
      \draw[cd-arrow-style] (GlA1) -- (GA1);
      \draw[cd-arrow-style,description,phantom] (GlA1) -- node[sloped,very near start] {\rotatebox{45}{$\lrcorner$}} (G1);
      \node (GlAA) at (0,2.2) {$\Gamma_{01}.l_0^*A_0.\pi_{l_0^*A_0}^*l_1^*A_1$};
      \draw[cd-arrow-style,fib] (GlAA) -- (GlA0);
      \draw[cd-arrow-style] (GlAA) -- (GlA1);
      \draw[cd-arrow-style,description,phantom] (GlAA) -- node[sloped,very near start] {\rotatebox{45}{$\lrcorner$}} (G*);
      \node (GA*) at (0,3.1) {$\Gamma_{01}.l_0^*A_0.\pi_{l_0^*A_0}^*l_1^*A_1.A_{01}$};
      \draw[cd-arrow-style,fib] (GA*) -- (GlAA);
      \draw[cd-arrow-style, bend right=20] (GA*) to (GA0);
      \draw[cd-arrow-style, bend left=20] (GA*) to (GA1);
    \end{tikzpicture}\]
  \item the reindexing of a type $(A_0,A_1,A_{01})$ along a map $(f_0,f_1,f_{01})$ as in the diagram above is taken to be $(f_0^*A_0,\, f_1^*A_1,\, (f_{01}.l_0^*A_0.\pi_{l_0^*A_0}^*l_1^*A_1)^* A_{01})$, with the connecting map $(f_0,f_1,f_{01}).(A_0,A_1,A_{01})$ taken as $(f_0.A_0,\, f_1.A_1,\, (f_{01}.l_0^*A_0.\pi_{l_0^*A_0}^*l_1^*A_1).A_{01})$.
  \end{enumerate}

  Then $\Eqv[\C]$ is the full sub-CwA of $\Span[\C]$ consisting of:
  \begin{enumerate}
  \item as objects, spans such that both legs $l_0$, $l_1$ are equivalences (in the sense of Definition~\ref{def:equiv-in-cwa});
  \item and as types over $\GGamma$, all types $\AA$ as above such that the resulting context extension $\GGamma.\AA$ is again a span-equivalence, or equivalently such that the maps
    \[ \Gamma_{01}.l_0^*A_0.\pi_{l_0^*A_0}^*l_1^*A_1.A_{01} \to \Gamma_{01}.l_i^*A_i \]
    are both equivalences.
  \end{enumerate}

  There are evident forgetful functors $P_0, P_1 : \Span[\C] \to \C$, taking a span to its left and right feet respectively;
  and since the structure on these components is defined pointwise, $P_0$ and $P_1$ are moreover maps of CwA's.
  We write $P_0$, $P_1$ also for the restruction of these CwA maps to $\Eqv[\C]$.
\end{definition}

\begin{remark}
  In more syntactic language, a closed type of $\Span[\C]$ consists of three closed types in $\C$:
  \[ 1 \types A_0 \type \qquad 1 \types A_1 \type \qquad x_0 \of A_0,\, x_1 \of A_1 \types A_{01} \type\]
  
  More generally, a type over a context $\GGamma$ consists of three types of the original model
  \[ \Gamma_0 \types A_0 \type \qquad \Gamma_1 \types A_1 \type \qquad \Gamma_{01},\, x_0 \of l_0^*A_0,\, x_1 \of l_1^*A_1 \types A_{01} \type \]
  and the context extension is the evident span of projections
  \[ \Gamma_0,\, x_0 \of A_0 \ \from\ \Gamma_{01},\, x_0 \of l_0^*A_0,\, x_1 \of l_1^*A_1,\, x_{01} \of A_{01}\ \to\ \Gamma_1,\, x_1 \of A_1. \]

  Such a span is a \defemph{(span\nobreakdash-)equivalence}---so lies in $\Eqv[\C]$---exactly if it additionally satisfies the judgements that the context extensions
  \begin{align*}
    \Gamma_{01},\, x_0 \of l_0^*A_0 & \types (x_1 \of l_1^*A_1,\, x_{01} \of A_{01}) \cxt \\
    \Gamma_{01},\, x_1 \of l_1^*A_1 & \types (x_0 \of l_0^*A_0,\, x_{01} \of A_{01}) \cxt
  \end{align*}
  are both contractible (where contractibility of context extensions is defined in the evident way using their identity contexts).
\end{remark}

\begin{remark}
  A closely related model is studied by Tonelli \cite{tonelli}.
  There, it is given syntactically, as the \defemph{relation} model of type theory.
  Precisely, Tonelli’s model may be seen as the contextual core of the CwA $\Span[\C_\T]$, where $\C_\T$ is the syntactic category of the type theory set out there.
\end{remark}

\begin{remark} \label{rmk:orderings-for-cxts}
  The orderings are used in the definition of types, where the context of $A_{01}$ is taken as $\Gamma_{01}.l_0^*A_0.\pi_{l_0^*A_0}^*l_1^*A_1$: here $0 <_{(01)} 1$ determines that $A_0$ is adjoined before $A_1$.
\end{remark}

\begin{remark}
  It may seem surprising that we use the mere \emph{property} of being an equivalence, rather than equipping the maps involved with data witnessing this.
  
  One certainly could try building a CwA of such structured equivalences (and that would obviate the need to use spans).
  However, the present approach seems to simplify many proofs and constructions, since everything fits into the general framework of homotopical inverse diagrams; for instance, all logical structure on $\Eqv[\C]$ is simply inherited from $\Span[\C]$.

  This approach also ensures that $\Eqv[\C]$ depends just on the class of equivalences in $\C$, not on the specific choice of $\Id$-types.
  This is not needed for the purposes of the present paper, but may (we expect) be useful in other applications.
\end{remark}

\begin{proposition} \label{prop:id-etc-in-eqv} \leavevmode
  \begin{enumerate}
  \item $\Eqv[\C]$ is naturally equipped with $\Id$-types;
  \item if $\C$ additionally carries $\synSigma$- types (resp.\ unit types) then so does $\Eqv[\C]$;
  \item if $\C$ has $\synPi$-types and functional extensionality, then so does $\Eqv[\C]$;
  \item moreover, in all these cases, the maps $P_i : \Eqv[\C] \to \C$ preserve such structure.    
   \end{enumerate}
\end{proposition}

\begin{proof}
  Immediate from Proposition~\ref{prop:deferred-results}.
\end{proof}

\begin{remark}
  The direct construction of the structure for Proposition~\ref{prop:id-etc-in-eqv} consists roughly of showing that each constructor preserves equivalences of types.
  This is why extensionality is required for the $\synPi$-types.
\end{remark}

\begin{proposition}  \label{prop:eqv-tfib}
  The evident map $\Eqv[\C] \to \C \times \C$ is a local fibration of CwA’s, preserving whatever logical structure is present. 
  Similarly, the maps $P_i : \Eqv[\C] \to \C$ are local trivial fibrations preserving the logical structure.
\end{proposition}

\begin{proof}
  Again, an immediate application of Proposition~\ref{prop:deferred-results}, noting for the second part that the inclusion of either $(0)$ or $(1)$ into $\EqvCat$ is a homotopy equivalence.
\end{proof}

\subsection{Reflexivity spans}

We would like to use $\Eqv[\C]$ as some kind of path object construction.
Most notions of “path object”, however, include at least a “reflexivity” map $\C \to P\C$ over the diagonal $\C \to \C \times \C$; and unfortunately, $\Eqv[\C]$ does not in general seem to admit such a map.\footnote{We do not know of any obstruction to the existence of such a map; but it seems unlikely to us that such maps exist in general.}
There is an evident functor on underlying categories, sending an object to the constant span on it; and this lifts suitably to a map on the presheaves of types, sending a type to its identity type span.
However, this commutes only laxly with context extension, and does not commute at all with the logical structure; so it does not define a map of CwA’s, let alone structured ones.

In lieu of a reflexivity map, therefore, we instead give reflexivity as a “weak map”; that is, a span whose left leg is a local trivial fibration:
\[ \begin{tikzcd}
  \EqvRefl[\C] \arrow[r] \arrow[d,tfib] & \Eqv[\C] \arrow[d] \\
  \C \arrow[r, fib, "\Delta"] & \C \times \C 
\end{tikzcd} \]

This suffices for the purposes of Section~\ref{sec:put-it-together} below (and for various other applications).

Roughly speaking, a type in $\EqvRefl[\C]$ consists of a type $A_0$ of $\C$ equipped with an auto-(span\nobreakdash-)equivalence $\parpair{A_*}{A_0}$ that is in some sense trivial, i.e.\ homotopic to the identity equivalence.

One’s first thought might be to express triviality of the auto-equivalence by a reflexivity map $r : A_0 \to A_*$ over $\Delta_A$.
However, this does not (it seems) yield a CwA; so once again, we replace this map by a weak map.

Precisely, $\EqvRefl[\C]$ is constructed as another CwA of homotopical inverse diagrams:

\begin{definition}
  $\EqvReflCat$ is the homotopical ordered inverse category
  \[ \begin{tikzcd} C \arrow[r, "p"] & * \arrow[r, shift left, "l_0"] \arrow[r, shift right, "l_1"'] & 0 \end{tikzcd} \qquad l_0 p = l_1 p \]
  with all maps equivalences, and $l_0 <_{*} l_1$.
  We write $lp$ for the common composite $l_0p = l_1p$.
\end{definition}

\begin{definition}
  $\EqvRefl[\C]$ is the CwA of homotopical diagrams on $\EqvReflCat$ in $\C$.
  We call such diagrams \defemph{trivial auto-(span\nobreakdash-)equivalences} in $\C$.
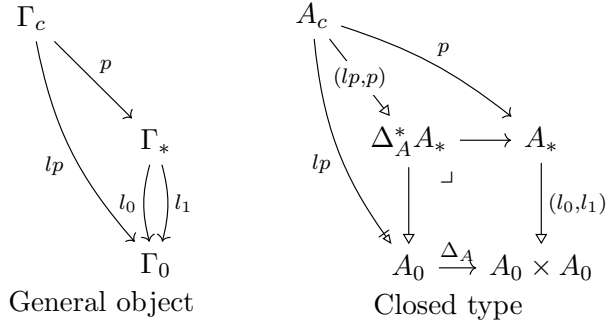
  \[
  \begin{tabular}{c}
    \begin{tikzcd}
      \Gamma_c \ar[dr,"p"] \ar[ddr,"l p"',bend right=15] & \\
      & \Gamma_* \ar[d,"l_0"',bend right=15] \ar[d,"l_1",bend left=15] \\
      & \Gamma_0
    \end{tikzcd}
    \\
    \text{General object}
  \end{tabular}
  \qquad
  \begin{tabular}{c}
    \begin{tikzcd}[column sep=tiny]
      A_c \ar[dr,fib,"{(lp,p)}" description] \ar[drr,"p", bend left=10] \ar[ddr,tfib,"lp"',bend right=15] & & \\
      & \Delta_A^* A_* \ar[r] \ar[d,fib] \arrow[dr, phantom, "\lrcorner", very near start]
      & A_* \ar[d,fib,"{(l_0,l_1)}"] \\ 
      & A_0 \ar[r,"\Delta_A"] & A_0 \times A_0 
    \end{tikzcd}
    \\
    \text{Closed type}
  \end{tabular}
  \]
\end{definition}

\begin{remark}
  In traditional type-theoretic notation, suppose $\GGamma$ is a homotopical diagram on $\EqvReflCat$:
  \[ \begin{tikzcd} 
    \Gamma_c \ar[r,"p"] & \Gamma_*  \ar[r,"l_0", bend left = 15] \ar[r,"l_1"', bend right = 15] & \Gamma_0
  \end{tikzcd} \]

  Then a type over $\GGamma$ in $\EqvRefl[\C]$ consists of types
  \begin{gather*}
    \Gamma_0 \types A_0 \type
    \qquad
    \Gamma_*,\, x_0 \of l_0^* A_0,\,  x_1 \of l_1^* A_0 \types A_*(x_0,x_1) \type
    \\
    \Gamma_c,\, x_0 \of (lp)^* A_0,\, x_* \of p^*A_*(x_0,x_0) \types A_c(x_0,x_*) \type
  \end{gather*}
  such that the following context extensions are contractible:
  \begin{align*}
    \Gamma_*,\, x_0 \of l_0^*A_0 & \types x_1 \of l_1^*A_0,\, x_* \of A_*(x_0,x_1) \cxt
    \\
    \Gamma_c,\, x_0 \of (lp)^* A_0 & \types x_* \of p^*A_*(x_0,x_0),\, x_c \of A_c(x_0,x_*) \cxt.
  \end{align*}

  (Contractibility of these contexts corresponds to $\Gamma.A$ sending $p_0$ and $lp$ to equivalences; this suffices for homotopicality since these maps generate the equivalences of $\EqvReflCat$ under 2-out-of-3.)
\end{remark}

\begin{example}
  Any type $A \in \Ty_\C(\Gamma)$ gives rise to a type in $\EqvRefl[\C]$ over the constant diagram on $\Gamma$:
  \begin{gather*}
    \Gamma \types A \type
    \qquad
    \Gamma,\, x_0,x_1 \of A \types \Id_A(x_0,x_1) \type
    \\
    \Gamma,\, x_0 \of A,\, x_* \of \Id_A(x_0,x_0) \types \Id_{\Id_A(x_0,x_0)}(x_*,\refl(x_0)) \type
  \end{gather*}
\end{example}

\begin{proposition} \label{prop:id-etc-in-eqvrefl} \leavevmode
  \begin{enumerate}
  \item $\EqvRefl[\C]$ carries $\Id$-types;
  \item if $\C$ additionally carries $\synSigma$- types (resp.\ unit types) then so does $\EqvRefl[\C]$;
  \item if $\C$ has $\synPi$-types and functional extensionality, then so does $\EqvRefl[\C]$;
  \item moreover, in all these cases, the natural map $\EqvRefl[\C] \to \Eqv[\C]$ preserves such structure.    
   \end{enumerate}
\end{proposition}

\begin{proof}
  Again, a direct application of Proposition~\ref{prop:deferred-results}.
\end{proof}

Finally, we show that $\EqvRefl[\C]$ gives a weak map from $\C$ to $\Eqv[\C]$ as promised.

\begin{proposition}  \label{prop:eqvrefl-tfib}
  The projection map $\ev_0 : \EqvRefl[\C] \to \C$ is a local trivial fibration.
\end{proposition}

\begin{proof}
  By Proposition~\ref{prop:deferred-results}, since the inclusion of $(0)$ in $\EqvReflCat$ is injective and a homotopy equivalence. 
\end{proof}

\subsection{Composites of spans}

As with reflexivity, one would hope for a “composition” map on span-equivalences, of the form:
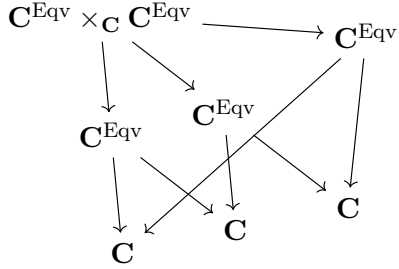
\[ \begin{tikzpicture}[x={(1.5,0.3)},y={(-0.6,1.6)},z={(1,-0.3)}]
\node (C0) at (0,0,0) {$\C$};
\node (C1) at (1,0,0) {$\C$};
\node (C2) at (2,0,0) {$\C$};
\node (C01) at (0.5,1,0) {$\Eqv[\C]$};
\draw[cd-arrow-style] (C01) -- (C0);
\draw[cd-arrow-style] (C01) -- (C1);
\node (C12) at (1.5,1,0) {$\Eqv[\C]$};
\draw[cd-arrow-style] (C12) -- (C1);
\draw[cd-arrow-style] (C12) -- (C2);
\node (C012) at (1,2,0) {$\Eqv[\C] \times_\C \Eqv[\C]$};
\draw[cd-arrow-style] (C012) -- (C01);
\draw[cd-arrow-style] (C012) -- (C12);
\node (C02) at (1,2,2) {$\Eqv[\C]$};
\draw[cd-arrow-style] (C02) -- (C0);
\draw[cd-arrow-style] (C02) -- (C2);
\draw[cd-arrow-style] (C012) -- (C02);
\end{tikzpicture} \]

Again, however, it seems difficult to define such a map in general, so we construct it as a weak map, i.e.\ a left-trivial span over $\C \times \C$:
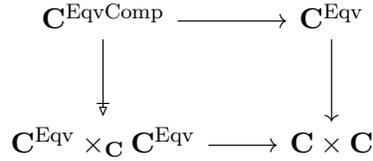
\[ \begin{tikzcd}
  \EqvComp[\C] \arrow[r] \arrow[d,tfib] & \Eqv[\C] \arrow[d] \\
  \Eqv[\C] \times_\C \Eqv[\C] \ar[r] & \C \times \C 
\end{tikzcd} \]

Roughly, an object of $\EqvComp[\C]$ should consist of a pair of “input” equivalences; an “output” equivalence; and a homotopy from the composite of the input pair to the output pair.
Translated entirely into span-equivalences, this becomes a diagram
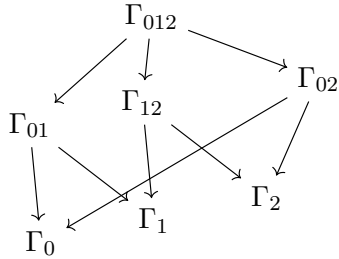
\[ \begin{tikzpicture}[x={(1.5,0.3)},y={(-0.6,1.6)},z={(1,-0.3)}]
\node (C0) at (0,0,0) {$\Gamma_0$};
\node (C1) at (1,0,0) {$\Gamma_1$};
\node (C2) at (2,0,0) {$\Gamma_2$};
\node (C01) at (0.5,1,0) {$\Gamma_{01}$};
\draw[cd-arrow-style] (C01) -- (C0);
\draw[cd-arrow-style] (C01) -- (C1);
\node (C12) at (1.5,1,0) {$\Gamma_{12}$};
\draw[cd-arrow-style] (C12) -- (C1);
\draw[cd-arrow-style] (C12) -- (C2);
\node (C02) at (1,1.5,2) {$\Gamma_{02}$};
\draw[cd-arrow-style] (C02) -- (C0);
\draw[cd-arrow-style] (C02) -- (C2);
\node (C012) at (1,2,1) {$\Gamma_{012}$};
\draw[cd-arrow-style] (C012) -- (C01);
\draw[cd-arrow-style] (C012) -- (C12);
\draw[cd-arrow-style] (C012) -- (C02);
\end{tikzpicture} \]
in which all maps are equivalences.
(Think of $\Gamma_{012}$ as a span from $\Gamma_{01} \times_{\Gamma_1} \Gamma_{12}$ to $\Gamma_{02}$ over $\Gamma_0 \times \Gamma_2$, but expressed in a way that doesn’t assume existence of that pullback.)

Flattened out, the domain of the above diagram is a familiar object: the category of faces of the 2-simplex.
Concretely,
\begin{definition}
  $\EqvCompCat$ is the posetal inverse category generated by the graph
  
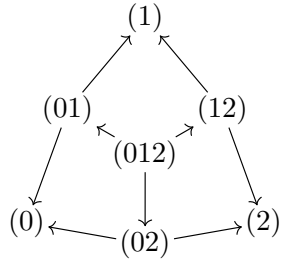
\[ \begin{tikzpicture}[scale=1.6,every node/.style={inner sep=1pt}]
    \path
    (0,0) node (0) {$(0)$}
    to ++(70:1) node (01) {$(01)$}
    to ++(50:1) node (1) {$(1)$}
    to ++(-50:1) node (12) {$(12)$}
    to ++(-70:1) node (2) {$(2)$}
    to ++(190:1) node (02) {$(02)$} ;
    \path [name path=aux0] (0.center) -- (12.center);
    \path [name path=aux1] (1.center) -- (02.center);
    \path [name intersections={of = aux0 and aux1}];
    \node (012) at (intersection-1) {$(012)$};
    
    \draw[cd-arrow-style] (012) -- (12);
    \draw[cd-arrow-style] (012) -- (02);
    \draw[cd-arrow-style] (012) -- (01);
    \draw[cd-arrow-style] (01) -- (0);
    \draw[cd-arrow-style] (01) -- (1);
    \draw[cd-arrow-style] (02) -- (0);
    \draw[cd-arrow-style] (02) -- (2);
    \draw[cd-arrow-style] (12) -- (1);
    \draw[cd-arrow-style] (12) -- (2);
  \end{tikzpicture} \]
  with all maps weak equivalences, and with objects of each coslice ordered according to their codomains under $(0) <(1) < (2)< (01) < (12) < (02) < (012)$.
\end{definition}

\begin{definition}
  $\EqvComp[\C]$ is the CwA of homotopical diagrams on $\EqvCompCat$ in $\C$.
\end{definition}

\begin{remark}
  In more traditional type-theoretic notation, suppose $\GGamma$ is a homotopical diagram on $\EqvCompCat$, with objects and maps denoted as e.g.\ $l_{0\hat{1}2} : \Gamma_{012} \to \Gamma_{02}$.
  Then a type over $\GGamma$ in $\EqvComp[\C]$ consists of types
  \begin{gather*}
    \Gamma_0 \types A_0 \type
    \qquad
    \Gamma_1 \types A_1 \type
    \qquad
    \Gamma_2 \types A_2 \type
    \\
    \begin{aligned}
    \Gamma_{01},\, x_0 \of l_{0\hat{1}}^*A_0,\,x_1 \of l_{\hat{0}1}^*A_1 & \types A_{01}(x_0,x_1) \type
    \\
    \Gamma_{12},\, x_1 \of l_{1\hat{2}}^*A_1,\,x_2 \of l_{\hat{1}2}^*A_2 & \types A_{12}(x_1,x_2) \type
    \\
    \Gamma_{02},\, x_0 \of l_{0\hat{2}}^*A_0,\,x_2 \of l_{\hat{0}2}^*A_2 & \types A_{02}(x_0,x_2) \type
    \end{aligned}
    \\
    \begin{multlined}[c][0.95\displaywidth]
      \Gamma_{012},\,
      x_0 \of l_{0\hat{1}\hat{2}}^*A_0,\, 
      x_1 \of l_{\hat{0}1\hat{2}}^*A_1,\, 
      x_2 \of l_{\hat{0}\hat{1}2}^*A_2,\,
      x_{01} \of l_{01\hat{2}}^*A_{01}(x_0,x_1),\, 
    \\
      x_{12} \of l_{\hat{0}12}^*A_{12}(x_1,x_2),\,
      x_{02} \of l_{0\hat{1}2}^*A_{02}(x_0,x_2)
      \types A_{012}(x_0,x_1,x_2,x_{01},x_{12},x_{02}) \type
      \mystrut{2.5ex} 
    \end{multlined}
  \end{gather*}
  such that the following context extensions are all contractible:
  \begin{gather*}
    \begin{aligned}
    \Gamma_{01},\, x_0 \of l_{0\hat{1}}^* A_0 & \types x_1 \of l_{\hat{0}1}^*A_1,\, x_{01} \of A_{01}(x_0,x_1) \cxt\\
    \Gamma_{01},\, x_1 \of l_{\hat{0}1}^* A_1 & \types x_0 \of l_{0\hat{1}}^*A_0,\, x_{01} \of A_{01}(x_0,x_1) \cxt \\
    \Gamma_{12},\, x_1 \of l_{1\hat{2}}^* A_1 & \types x_2 \of l_{\hat{1}2}^*A_2,\, x_{12} \of A_{12}(x_1,x_2) \cxt \\
    \Gamma_{12},\, x_2 \of l_{\hat{1}2}^* A_2 & \types x_1 \of l_{1\hat{2}}^*A_1,\, x_{12} \of A_{12}(x_1,x_2) \cxt \\
    \Gamma_{02},\, x_0 \of l_{0\hat{2}}^* A_0 & \types x_2 \of l_{\hat{0}2}^*A_2,\, x_{02} \of A_{02}(x_0,x_2) \cxt 
    \end{aligned} \\
    \begin{multlined}[c][0.95\displaywidth]
      \Gamma_{012},\,
      x_0 \of l_{0\hat{1}\hat{2}}^*A_0,\, 
      x_1 \of l_{\hat{0}1\hat{2}}^*A_1,\, 
      x_2 \of l_{\hat{0}\hat{1}2}^*A_2,\,
      x_{01} \of l_{01\hat{2}}^*A_{01}(x_0,x_1),\,
      x_{12} \of l_{\hat{0}12}^*A_{12}(x_1,x_2) \\
      \types
      x_{02} \of l_{0\hat{1}2}^*A_{02}(x_0,x_2),\ 
      x_{012} \of A_{012}(x_0,x_1,x_2,x_{01},x_{12},x_{02}) \cxt
      \mystrut{2.5ex} 
    \end{multlined}
  \end{gather*}
(Again, these form a minimal subclass ensuring that all maps in the resulting diagram $\Gamma.A$ are equivalences.)
\end{remark}

\begin{proposition} \label{prop:id-etc-in-eqvcomp} \leavevmode
  \begin{enumerate}
  \item $\EqvComp[\C]$ carries $\Id$-types;
  \item if $\C$ additionally carries $\synSigma$- types (resp.\ unit types) then so does $\EqvComp[\C]$;
  \item if $\C$ has $\synPi$-types and functional extensionality, then so does $\EqvComp[\C]$;
  \item moreover, in all these cases, the evident maps $P_{01}, P_{12}, P_{02} : \EqvComp[\C] \to \Eqv[\C]$ all preserve such structure.    
  \end{enumerate}
\end{proposition}

\begin{proof}
  Once again, a direct application of Proposition~\ref{prop:deferred-results}.
\end{proof}

Finally, we once again must show that $\EqvComp[\C]$ can be viewed as a weak map as intended.

\begin{proposition} \label{prop:eqvcomp-tfib}
  The map $(P_{01},P_{12}) : \EqvComp[\C] \to \Eqv[\C] \times_\C \Eqv[\C]$ is a local trivial fibration.
\end{proposition}

\begin{proof}
  By Proposition~\ref{prop:deferred-results}, since $\Eqv[\C] \times_\C \Eqv[\C] \iso \C^{\EqvCat +_{1} \EqvCat}_h$, and the inclusion of $\EqvCat +_{1} \EqvCat$ into $\EqvCompCat$ is a homotopy equivalence.
\end{proof}

\subsection{Inverse span-equivalences}

One might think that symmetry is simpler than reflexivity and transitivity, since the functor $\Eqv[\C] \to \Eqv[\C]$ switching the direction of spans extends to an evident action on types.
However, it is not quite a CwA map---it preserves context extension only up to isomorphism---and may preserve logical structure only up to equivalence.
These problems stem from the fact that the switch automorphism $\EqvCat \to \EqvCat$ does not respect the ordering $<_{(01)}$, so the functoriality of $\C^{(-)}$ (Proposition~\ref{prop:deferred-results}(\ref{item:contravariance})) does not act on this automorphism.

We thus once again use a diagram category, taking more care than usual in the choice of ordering, to obtain a symmetry “weak map” $\Eqv[\C] \to \Eqv[\C]$.

\begin{definition}
  $\EqvInvCat$ is the posetal inverse category
  \[ \begin{tikzcd}[row sep=small]
    & (*) \ar[dl] \ar[dr] \\
    (01) \ar[d] \ar[drr] & & (10) \ar[d] \ar[dll] \\
    (0) & & (1)
  \end{tikzcd} \]
  with all maps equivalences, and with orderings given by
\[ (0) <_{(01)} (1), \qquad \qquad (1) <_{(10)} (0), \qquad \qquad (0) <_{(*)} (1) <_{(*)} (01) <_{(*)} (10). \]

  Write $P_{01} : \EqvCat \to \EqvInvCat$ for the evident inclusion functor, and $P_{10} : \EqvCat \to \EqvInvCat$ for the functor interchanging $(0)$ and $(1)$ and sending $(01)$ to $(10)$.
  The chosen orderings (in particular, $<_{(10)}$) ensure that both $P_{01}$ and $P_{10}$ are order-preserving discrete opfibrations.
\end{definition}

\begin{definition}
  $\EqvInv[\C]$ is the CwA of homotopical diagrams on $\EqvInvCat$ in $\C$.
  Such a diagram can be viewed as a pair of mutually inverse span-equivalences.
\end{definition}

\begin{remark}
  In type-theoretic notation, let $\GGamma$ be a homotopical diagram on $\EqvInvCat$, with objects and maps denoted as e.g.\ $l_{1\hat{0}} : \Gamma_{10} \to \Gamma_{1}$, and $m_{10} : \Gamma_* \to \Gamma_{10}$.

  Then a type over $\GGamma$ in $\EqvInv[\C]$ consists of types
  \begin{gather*}
    \Gamma_0 \types A_0 \type
    \qquad
    \Gamma_1 \types A_1 \type
    \\
    \begin{aligned}
    \Gamma_{01},\, x_0 \of l_{0\hat{1}}^* A_0,\,  x_1 \of l_{\hat{0}1}^* A_1 & \types A_{01}(x_0,x_1) \type
    \\
    \Gamma_{10},\, x_1 \of l_{1\hat{0}}^* A_1,\,  x_0 \of l_{\hat{1}0}^* A_0 & \types A_{10}(x_1,x_0) \type
    \end{aligned}
    \\
    \begin{multlined}[c][0.8\displaywidth]
      \Gamma_{*},\,
      x_0 \of m_{0}^* A_0,\, 
      x_1 \of m_{1}^* A_1,\,
      x_{01} \of m_{01}^*A_{01}(x_0,x_1),\,
      x_{10} \of m_{10}^*A_{10}(x_1,x_0)
      \\
      \types A_{*}(x_0,x_1,x_{01},x_{10}) \type
      \mystrut{2.5ex} 
    \end{multlined}
  \end{gather*}
  such that the following context extensions are contractible:
  \begin{gather*}
    \begin{aligned}
    \Gamma_{01},\, x_0 \of l_{0\hat{1}}^* A_0 & \types x_1 \of l_{\hat{0}1}^* A_1,\, x_{01} \of A_{01}(x_0,x_1) \cxt
    \\
    \Gamma_{01},\, x_1 \of l_{\hat{0}1}^* A_1 & \types x_0 \of l_{0\hat{1}}^* A_0,\, x_{01} \of A_{01}(x_0,x_1) \cxt
    \\
    \Gamma_{10},\, x_1 \of l_{1\hat{0}}^* A_1 & \types x_0 \of l_{\hat{1}0}^* A_0,\, x_{10} \of A_{10}(x_1,x_0) \cxt
    \\
    \Gamma_{10},\, x_0 \of l_{\hat{1}0}^* A_0 & \types x_1 \of l_{1\hat{0}}^* A_1,\, x_{10} \of A_{10}(x_1,x_0) \cxt
    \end{aligned}
    \\
    \begin{multlined}[c][0.8\displaywidth]
      \Gamma_{*},\, x_0 \of m_{0}^* A_0,\,  x_1 \of m_{1}^* A_1,\, x_{01} \of m_{01}^*A_{01}(x_0,x_1) 
      \\
      \types x_{10} \of m_{10}^*A_{10}(x_1,x_0),\, x_* \of A_{*}(x_0,x_1,x_{01},x_{10}) \cxt
      \mystrut{2.5ex} 
    \end{multlined}
  \end{gather*}
  (Once again, these extensions suffice since their projections generate the equivalences of $\EqvInvCat$ under 2-out-of-3.)
\end{remark}

\begin{proposition} \label{prop:id-etc-in-eqvinv} \leavevmode
  \begin{enumerate}
  \item $\EqvInv[\C]$ carries $\Id$-types;
  \item if $\C$ additionally carries $\synSigma$- types (resp.\ unit types) then so does $\EqvInv[\C]$;
  \item if $\C$ has $\synPi$-types and functional extensionality, then so does $\EqvInv[\C]$;
  \item in all these cases, the maps $P_{01}^*, P_{10}^* : \EqvInv[\C] \to \Eqv[\C]$ preserves the structure;
  \item moreover, $P_{01}^*$ and $P_{10}^*$ are local trivial fibrations.
   \end{enumerate}
\end{proposition}

\begin{proof}
  For the most part, a direct application of Proposition~\ref{prop:deferred-results}.
  For the last item, note that $P_{01}, P_{10} : \EqvInvCat \to \EqvCat$ are homotopy equivalences, since $\EqvCat$ and $\EqvInvCat$ both have initial objects and all maps equivalences.
\end{proof}

\begin{remark}
  Astute readers may notice that the final propositions of the subsections above have effectively shown:
  \begin{enumerate}
  \item $\Eqv[\C]$ forms a Reedy span-equivalence from $\C$ to itself;
  \item $\Eqv[\C]$ together with $\EqvRefl[\C]$ forms a trivial auto-equivalence of $\C$;
  \item $\Eqv[\C]$ together with $\EqvComp[\C]$ forms a commuting triangle of equivalences;
  \item $\Eqv[\C]$ together with $\EqvInv[\C]$ form a mutually inverse pair of auto-equivalences.
  \end{enumerate}
  The authors did not notice this until quite late in the preparation of this article.
\end{remark}


\section{The left semi model structure on contextual categories}
\label{sec:put-it-together}

We now have all the main ingredients prepared to deduce that the three classes of maps introduced in Section \ref{sec:classes-of-maps} form a left semi-model structure.

In this section, we first bring the span-equivalences construction back to the contextual world, and use it to define homotopy between maps of contextual categories.
We then recall the definition of left semi-model structure, and show with just a little diagram chasing that we have one on our hands.

\subsection{Returning to the contextual world}

The CwA’s $\Eqv$, $\EqvRefl$ and $\EqvComp$ of the previous section will almost never be contextual.
To bring them back to the contextual setting, we take their cores.

Making liberal use of Proposition~\ref{prop:classes-and-core} (that $\core$ sends local (trivial) fibrations to (trivial) fibrations), together with the fact that core is a coreflection (so it preserves limits, and $\core \C \iso \C$ when $\C$ is contextual), we sum up the result:

\begin{proposition}
  For each $\C$ in $\CxlCat_{\Id,\unit,\synSigma(,\Piext)}$, we have diagrams as follows, all in $\CxlCat_{\Id,\unit,\synSigma(,\Piext)}$, and functorial in $\C$:

  \begin{gather*}
    \begin{tikzcd}[ampersand replacement=\&]
      \& \core \Eqv[\C] \ar[dl,tfib] \ar[d,fib] \ar[dr,tfib] \& \\
      \C \& \C \times \C \ar[l,fib] \ar[r,fib] \& \C
    \end{tikzcd}
    \qquad
    \begin{tikzcd}[ampersand replacement=\&]
      \core \EqvRefl[\C] \arrow[r,fib] \arrow[d,tfib] \& \core \Eqv[\C] \arrow[d,fib] \\
      \C \arrow[r,"\Delta"] \& \C \times \C 
    \end{tikzcd}
    \\
    \begin{tikzcd}[ampersand replacement=\&]
      \core \EqvComp[\C] \arrow[r,fib, "P_{02}"] \arrow[d,tfib, "{(P_{01},P_{12})}"'] \& \core \Eqv[\C] \arrow[d,fib] \\
      \core \Eqv[\C] \times_{\core \C} \core \Eqv[\C] \arrow[r,"{(P_0,P_2)}",fib] \& \C \times \C 
    \end{tikzcd}
  \end{gather*}
  \vskip -1.7em \thmqed 
\end{proposition}

For readability, for the remainder of this section, we will omit the “$\core$” and write just $\Eqv[\C]$ and so on, since we have no further need of the CwA versions.

\subsection{The right homotopy relation}

Using $\Eqv$ as a path-object construction, we can define a notion of right homotopy between maps in $\CxlCat_{\Id,\unit,\synSigma(,\Piext)}$, which will be well-behaved under cofibrant domains.

\begin{definition}
  Say $F_0, F_1 : \C \to \D$ in $\CxlCat_{\Id,\unit,\synSigma(,\Piext)}$ are \defemph{right homotopic} ($F_0 \rhomot F_1$) if they factor jointly through $\Eqv[\D]$:
    \[\begin{tikzpicture}[cd-style,x={(1cm,0cm)},y={(0cm,1.5cm)},z={(0.8cm,-0.35cm)}]
      \node (DxD) at (3,0,0) {$\D \times \D$};
      \node (EqvD) at (3,1,0) {$\Eqv[\D]\!\!$};
      \draw[cd-arrow-style,fib] (EqvD) -- node {$(P_0,P_1)$}  (DxD);
      \node (C) at (0,0,0) {$\C$};
      \draw[cd-arrow-style] (C) -- node {$(F_0,F_1)$} (DxD);
      \draw[cd-arrow-style,bend left=15] (C) to node {$H$} (EqvD);
    \end{tikzpicture}\]
\end{definition}

\begin{proposition}  \label{prop:homot-equiv-rel}
  When $\C$ is cofibrant, right homotopy is an equivalence relation on $\CxlCat_{\Id,\unit,\synSigma(,\Piext)}(\C,\D)$.
\end{proposition}

\begin{proof}
  Reflexivity: by Proposition \ref{prop:eqvrefl-tfib} and cofibrancy of $\C$, any map $F : \C \to \D$ lifts to a map $\C \to \EqvRefl[\D]$; composing this with the forgetful map $\EqvRefl[\D] \to \Eqv[\D]$ yields a reflexivity homotopy for $F$:
  \[\begin{tikzpicture}[cd-style,x={(1.2cm,0cm)},y={(0cm,1.5cm)},z={(0.8cm,-0.35cm)}]
    \node (C) at (0,0,0) {$\C$};
    \node (D) at (2,0,0) {$\D$};
    \node (reflD) at (2,1,0) {$\ \EqvRefl[\D]\!$};
    \draw[cd-arrow-style,tfib] (reflD) to (D); 
    \draw[cd-arrow-style] (C) -- node {$F$} (D);
    \draw[cd-arrow-style,dashed] (C) -- (reflD);
    \node (DxD) at (4,0,0) {$\D \times \D$};
    \node (EqvD) at (4,1,0) {$\Eqv[\D]\!\!$};
    \draw[cd-arrow-style,fib] (EqvD) -- node {$(P_0,P_1)$}  (DxD);
    \draw[cd-arrow-style] (D) -- node {$\Delta_D$} (DxD);
    \draw[cd-arrow-style] (reflD) to (EqvD);
  \end{tikzpicture}\]
  
  Symmetry is similar.
  If $H : \C \to \Eqv[\D]$ witnesses $F_0 \rhomot F_1$, then by Proposition~\ref{prop:id-etc-in-eqvinv} and cofibrancy of $\C$, $H$ lifts along $P_{01}^*$ to a map $\C \to \EqvInv[\D]$, whose composition with $P_{10}^*$ gives a homotopy $F_1 \rhomot F_0$.
  \[\begin{tikzpicture}[cd-style,x={(1.2cm,0cm)},y={(0cm,1.5cm)},z={(0.8cm,-0.35cm)}]
    \node (C) at (0,0,0) {$\C$};
    \node (EqvD) at (2,0,0) {$\Eqv[\D]$};
    \node (EqvInvD) at (2,1,0) {$\ \EqvInv[\D]\!$};
    \draw[cd-arrow-style,tfib] (EqvInvD) to node {$P_{01}^*$} (EqvD); 
    \draw[cd-arrow-style] (C) -- node {$H$} (EqvD);
    \draw[cd-arrow-style,dashed] (C) -- (reflD);
    \node (DxD) at (4,0,0) {$\D \times \D$};
    \node (EqvD') at (4,1,0) {$\Eqv[\D]\!\!$};
    \draw[cd-arrow-style,fib] (EqvD') -- node {$(P_0,P_1)$}  (DxD);
    \draw[cd-arrow-style] (EqvD) -- node {$(P_1,P_0)$} (DxD);
    \draw[cd-arrow-style] (EqvInvD) -- node {$P_{10}^*$} (EqvD');
  \end{tikzpicture}\]
  
  Transitivity is similar again.
  Given $F_0, F_1, F_2 : \C \to \D$, and homotopies $H_{01}$, $H_{02}$, we have an induced map $(H_{01},H_{02}) : \C \to \Eqv[\D] \times_\D \Eqv[\D]$.
  By Proposition \ref{prop:eqvcomp-tfib} and cofibrancy of $\C$, we can lift this to a map $\C \to \EqvComp[\D]$; composing this with $P_{02} : \EqvComp[\D] \to \Eqv[\D]$ gives a homotopy $F_0 \rhomot F_2$.
\end{proof}

\begin{proposition}
  Right homotopy is stable under pre- and post-composition.
\end{proposition}

\begin{proof}
  Let $H : \C \to \Eqv[\D]$ be a homotopy $F_0 \rhomot F_1$.
  Then for any $G : \C' \to \C$, $HG$ is a homotopy $F_0G \rhomot F_1G$;
  and similarly, for any $K : \D \to \D'$, $\Eqv[K]H$ is a homotopy $KF_0 \rhomot KF_1$.
\end{proof}

\begin{proposition} \label{prop:homot-to-equiv-is-equiv}
  Any map right homotopic to an equivalence is an equivalence.
\end{proposition}

\begin{proof}
  By Proposition~\ref{prop:eqv-tfib}, the maps $P_i : \Eqv[\D] \to \D$ are equivalences.
  So by 2-out-of-3, if $H : \C \to \Eqv[\D]$ is a homotopy $F_0 \rhomot F_1$, we have that $H$ is an equivalence if and only if each/either $F_i = P_i H$ is one.
\end{proof}

\subsection{Putting it all together: the semi model structure}

Finally, we show that the classes of maps on $\CxlCat_{\Id,\unit,\synSigma(,\Piext)}$ fit together to form a \emph{left semi-model structure}.\footnote{We do not know whether it also forms a full model structure; we have no specific obstruction or counterexample.}

Roughly, this means three classes of maps as in a model structure, except that the $(\Cof \cap \WEq,\Fib)$ factorization system only works for maps with cofibrant domains.

\begin{definition}[cf.\ {\cite[Def.\ 1(I)]{spitzweck:operads-algebras-modules}}]
  A \defemph{left semi-model structure} on a bicomplete category $\catE$ consists of three classes of maps: $\WEq$, $\Fib$, $\Cof$, subject to the axioms:
 \begin{enumerate}
  \item all three classes are closed under retracts; $\WEq$ satisfies the 2-out-of-3 property; and fibrations and trivial fibrations are preserved under pullback;
  \item cofibrations have the left lefting property with respect to trivial fibrations; and trivial cofibrations with cofibrant source have the left lifting property with respect to fibrations;
  \item every map can be functorially factored into a cofibration followed by a trivial fibration; every map with cofibrant source can also be functorially factored into a trivial cofibration followed by a fibration.
 \end{enumerate}
\end{definition}

(Left semi-model categories first appeared in Hovey \cite[Thm.\ 3.3]{hovey:monoidal-model}, and were further developed by Spitzweck \cite[Def.\ 1]{spitzweck:operads-algebras-modules} and Barwick \cite[Def.\ 1.4]{barwick:left-and-right-model-cats}.)

In practice, one usually has just a little more structure:

\begin{lemma} \label{lem:semi-model-strux-from-wfss}
  Suppose $\catE$ is a bicomplete category, equipped with
  \begin{enumerate}
  \item a class of maps $\WEq$, including all identities, and closed under 2-out-of-6 and retracts;
  \item two weak factorization systems $(\Anod, \Fib)$ and $(\Cof, \TFib)$;
  \item such that $\TFib = \Fib \cap \WEq$, and
  \item when $A \in \catE$ is cofibrant (i.e.~the map $0 \to A$ is in $\Cof$), a map $i : A \to B$ is in $\Anod$ if and only if it is in $\Cof \cap \WEq$.
  \end{enumerate}

  Then the classes $(\WEq,\Cof,\Fib)$ form a left semi-model structure on $\catE$. \thmqed
\end{lemma}

For the remainder of this section, we fix the ambient category as either $\CxlCat_{\Id,\unit,\synSigma}$ or $\CxlCat_{\Id,\unit,\synSigma,\Piext}$ (the two cases are exactly parallel), and work to establish the hypotheses of Lemma~\ref{lem:semi-model-strux-from-wfss}.

\begin{proposition} \label{prop:anod-iff-cof-and-weq}
  Let $\C$ be cofibrant in $\CxlCat_{\Id,\unit,\synSigma(,\Piext)}$.
  Then a map $F : \C \to \D$ in $\CxlCat_{\Id,\unit,\synSigma(,\Piext)}$ is anodyne precisely if it is both a weak equivalence and a cofibration.
\end{proposition}

\begin{proof}
  $\Anod \subseteq \Cof$: this does not require the cofibrant domain assumption.
  We noted above that $\TFib \subseteq \Fib$, so dually, $\Anod \subseteq \Cof$. \\

  $\Anod \subseteq \WEq$: suppose $j : \A \to \B$ is anodyne, with $\A$ cofibrant.
  By fibrancy of $\A$ (Proposition~\ref{prop:all-objects-fibrant}), we can take a retraction $r$ for $j$:
  \[\begin{tikzcd}
    \A \ar[r, "1_A"] \ar[d, "j"'] &
    \A \ar[d,fib] \\
    \B \ar[r] \ar[ur, dashed, "r"] &
    1
  \end{tikzcd}\]

  But then $r$ is also a homotopy section for $j$, by filling the square
  \[\begin{tikzcd}
    \A \ar[r, "c_j"] \ar[d, "j"'] &
    \Eqv[\B] \ar[d,fib] \\
    \B \ar[r, "{(1_B,jr)}"] \ar[ur, dashed, "H"] &
    \B \times \B
  \end{tikzcd}\]
  where $c_j$ is a reflexivity homotopy on $j$, supplied by Proposition~\ref{prop:homot-equiv-rel} since $\A$ is cofibrant.

  So $rj = 1_A$, and $jr \rhomot 1_B$; so by 2-out-of-6 and Proposition~\ref{prop:homot-to-equiv-is-equiv}, $j$  is an equivalence. \\

  $\WEq \cap \Cof \subseteq \Anod$: suppose $j : \A \to \B$ is in $\WEq$ and $\Cof$, with $\A$ cofibrant.

  We want to show that $j$ is orthogonal to all fibrations.
  It is enough to show this for fibrations over $\B$, since any other lifting problem can first be pulled back to $\B$.
  So assume $p : \Y \to \B$ is some fibration, with a map $h : \A \to \Y$ over $\B$; we want to fill the square
  \[\begin{tikzcd}
    \A \ar[r, "h"] \ar[d, "j"'] &
    \Y \ar[d, fib, "p"] \\
    \B \ar[r, "1_B"] &
    \B.
  \end{tikzcd}\]

  Proposition \ref{prop:homot-equiv-rel} gives a reflexivity homotopy $c_h : \A \to \Eqv[\Y]$ for $h$.
  Write $\Eqv[(h,1_\Y)]$ for the pullback $(h,1_\Y)^* \Eqv[\Y]$, with projection maps $(Q_0,Q_1) : \Eqv[(h,1_\Y)] \to \A \times \Y$.
  Then $c_h$ factors through $\Eqv[(h,1_\Y)]$ by a map $c'_h$:
  \[\begin{tikzcd}[column sep=large]
    & \Eqv[(h,1_\Y)] \ar[d,fib, "{(Q_0,Q_1)}" description] \ar[r] \arrow[dr, phantom, "\lrcorner", very near start]
    & \Eqv[\Y] \ar[d,fib, "{(P_0,P_1)}" description]
    \\ 
    \A \ar[ur, "c'_h"] \ar[r, "{(1_\A,h)}"] \ar[dr, "1_A"']
    & \A \times \Y \ar[r, "{(h,1_\Y)}"] \ar[d,fib, "\pi_0"] \arrow[dr, phantom, "\lrcorner", very near start]
    & \Y \times \Y \ar[d,fib, "\pi_0"]
    \\
    & \A \ar[r, "h"]
    & \Y
  \end{tikzcd}\]

  Now $Q_0 : \Eqv[(h,1_\Y)] \to \A$ is a pullback of $P_0 : \Eqv[\Y] \to \Y$; so by Proposition \ref{prop:eqv-tfib}, it is a trivial fibration.
  So by 2-of-3, $c'_h$ is an equivalence, since $Q_0 c'_h = 1_\A$,
  and by 2-of-3 again, so is $pQ_1$, since $pQ_1c'_h = ph = j$.

  But $pQ_1$ is also a fibration (as a composite of two fibrations); so by Proposition~\ref{prop:tfib-iff-fib-and-weq}, it is a trivial fibration.
  So (since $j$ is a cofibration) we can extend $c'_h$ along $j$, filling the left-hand square below; composing the resulting filler with $Q_1$ then solves the original lifting problem.
  \[
  \begin{tikzcd}
    \A \ar[r, "c'_h"'] \ar[d, "j"'] \ar[rr, bend left = 20, "h"]
    & \Eqv[(h,1_\Y)] \ar[r, "Q_1"'] \ar[d,tfib, "pQ_1"]
    & \Y \ar[d, fib, "p"]
    \\ \B \ar[r, "1_\B"] \ar[ur, dashed]
    & \B \ar[r, "1_\B"]
    & \B  
  \end{tikzcd}
  \]
\end{proof}

This completes the main result:

\begin{theorem}
  On each of $\CxlCat_{\Id,\unit,\synSigma}$ and $\CxlCat_{\Id,\unit,\synSigma,\Piext}$, the classes $\WEq$, $\Fib$, $\Cof$ of Section \ref{sec:classes-of-maps} form a left semi-model structure.
\end{theorem}

\begin{proof}
  Propositions~\ref{prop:tfib-iff-fib-and-weq} and \ref{prop:anod-iff-cof-and-weq} supply the hypotheses of Proposition~\ref{lem:semi-model-strux-from-wfss}.
\end{proof}

\begin{corollary}
  The cofibrant objects of $\CxlCat_{\Id,\unit,\synSigma(,\Piext)}$ form a cofibration category. \thmqed
\end{corollary}









\arxivorjournal{
  \bibliographystyle{amsalphaurlmod}
}{
  \bibliographystyle{elsarticle-num}
}

\bibliography{general-bibliography}

\providecommand{\bysame}{\leavevmode\hbox to3em{\hrulefill}\thinspace}
\providecommand{\MR}{\relax\ifhmode\unskip\space\fi MR }
\providecommand{\MRhref}[2]{%
  \href{http://www.ams.org/mathscinet-getitem?mr=#1}{#2}
}
\providecommand{\href}[2]{#2}
\begin{thebibliography}{{Uni}13}

\bibitem[AKL15]{avigad-kapulkin-lumsdaine}
Jeremy Avigad, Krzysztof Kapulkin, and Peter~LeFanu Lumsdaine, \emph{Homotopy
  limits in type theory}, Math. Structures Comput. Sci. \textbf{25} (2015),
  no.~5, 1040--1070, \href {http://arxiv.org/abs/1304.0680}
  {\path{arXiv:1304.0680}}, \href {http://dx.doi.org/10.1017/S0960129514000498}
  {\path{doi:10.1017/S0960129514000498}}.

\bibitem[Bar10]{barwick:left-and-right-model-cats}
Clark Barwick, \emph{On left and right model categories and left and right
  {B}ousfield localizations}, Homology, Homotopy and Applications, Vol. 12
  (2010), No. 2, pp.245-320 \textbf{12} (2010), no.~2, 245--320,
  \url{http://projecteuclid.org/euclid.hha/1296223884}.

\bibitem[Bar16]{barwick:k-theory-of-higher-cats}
\bysame, \emph{On the algebraic {$K$}-theory of higher categories}, J. Topol.
  \textbf{9} (2016), no.~1, 245--347, \href {http://arxiv.org/abs/1204.3607}
  {\path{arXiv:1204.3607}}, \href {http://dx.doi.org/10.1112/jtopol/jtv042}
  {\path{doi:10.1112/jtopol/jtv042}}.

\bibitem[BK12]{barwick-kan:characterization}
Clark Barwick and Daniel~M. Kan, \emph{A characterization of simplicial
  localization functors and a discussion of {DK} equivalences}, Indag. Math.
  (N.S.) \textbf{23} (2012), no.~1-2, 69--79, \href
  {http://dx.doi.org/10.1016/j.indag.2011.10.001}
  {\path{doi:10.1016/j.indag.2011.10.001}}.

\bibitem[Bro73]{brown:abstract-homotopy-theory}
Kenneth~S. Brown, \emph{Abstract homotopy theory and generalized sheaf
  cohomology}, Trans. Amer. Math. Soc. \textbf{186} (1973), 419--458.

\bibitem[Car78]{cartmell:thesis}
John Cartmell, \emph{Generalised algebraic theories and contextual categories},
  Ph.D. thesis, Oxford, 1978.

\bibitem[CD11]{clairmbault-dybjer:biequivalence}
Pierre Clairambault and Peter Dybjer, \emph{The biequivalence of locally
  {C}artesian closed categories and {M}artin-{L}\"of type theories}, Typed
  lambda calculi and applications, Lecture Notes in Comput. Sci., vol. 6690,
  Springer, Heidelberg, 2011, pp.~91--106, \href
  {http://arxiv.org/abs/1112.3456} {\path{arXiv:1112.3456}}, \href
  {http://dx.doi.org/10.1007/978-3-642-21691-6_10}
  {\path{doi:10.1007/978-3-642-21691-6_10}}.

\bibitem[Gar09]{garner:2-d-models}
Richard Garner, \emph{Two-dimensional models of type theory}, Math. Structures
  Comput. Sci. \textbf{19} (2009), no.~4, 687--736, \href
  {http://arxiv.org/abs/0808.2122} {\path{arXiv:0808.2122}}, \href
  {http://dx.doi.org/10.1017/S0960129509007646}
  {\path{doi:10.1017/S0960129509007646}}.

\bibitem[Hof95]{hofmann:thesis}
Martin Hofmann, \emph{Extensional concepts in intensional type theory}, Ph.D.
  thesis, University of Edinburgh, 1995,
  \url{http://www.lfcs.inf.ed.ac.uk/reports/95/ECS-LFCS-95-327/}.

\bibitem[Hov98]{hovey:monoidal-model}
Mark Hovey, \emph{Monoidal model categories}, preprint, 1998, \href
  {http://arxiv.org/abs/math/9803002} {\path{arXiv:math/9803002}}.

\bibitem[Hov99]{hovey:book}
\bysame, \emph{Model categories}, Mathematical Surveys and Monographs, vol.~63,
  American Mathematical Society, Providence, RI, 1999.

\bibitem[Isa17]{isaev:model-structure}
Valery Isaev, \emph{Model structures on categories of models of type theories},
  Math. Structures Comput. Sci. (2017), 1--28, online version, \href
  {http://arxiv.org/abs/1607.07407} {\path{arXiv:1607.07407}}, \href
  {http://dx.doi.org/10.1017/S0960129517000202}
  {\path{doi:10.1017/S0960129517000202}}.

\bibitem[Joy08]{joyal:notes-on-quasicategories}
Andr{\'e} Joyal, \emph{Notes on quasicategories}, unpublished notes, 2008,
  \url{http://www.math.uchicago.edu/~may/IMA/Joyal.pdf}.

\bibitem[Joy11]{joyal:remarks-on-homotopical-logic}
\bysame, \emph{Remarks on homotopical logic}, Mini-Workshop: The Homotopy
  Interpretation of Constructive Type Theory (Steve Awodey, Richard Garner, Per
  Martin-L\"{o}f, and Vladimir Voevodsky, eds.), vol.~8, Oberwolfach Reports,
  no.~1, European Mathematical Society, 2011, pp.~627--630, \href
  {http://dx.doi.org/10.4171/OWR/2011/11} {\path{doi:10.4171/OWR/2011/11}}.

\bibitem[Joy17]{joyal:tribes}
\bysame, \emph{Notes on clans and tribes}, preprint, 2017, \href
  {http://arxiv.org/abs/1710.10238} {\path{arXiv:1710.10238}}.

\bibitem[Kap17]{kapulkin:locally-cartesian-qcat}
Krzysztof Kapulkin, \emph{Locally cartesian closed quasi-categories from type
  theory}, Journal of Topology \textbf{10} (2017), no.~4, 1029--1049, \href
  {http://arxiv.org/abs/1507.02648} {\path{arXiv:1507.02648}}, \href
  {http://dx.doi.org/10.1112/topo.12031} {\path{doi:10.1112/topo.12031}}.

\bibitem[KL12]{kapulkin-lumsdaine:simplicial-model}
Krzysztof Kapulkin and Peter~LeFanu Lumsdaine, \emph{The simplicial model of
  {Univalent} {Foundations} (after {Voevodsky})}, preprint, 2012, \href
  {http://arxiv.org/abs/1211.2851} {\path{arXiv:1211.2851}}.

\bibitem[KL18]{kapulkin-lumsdaine:inverse-diagrams}
\bysame, \emph{Homotopical inverse diagrams in categories with attributes},
  preprint, 2018, \href {http://arxiv.org/abs/1808.01816}
  {\path{arXiv:1808.01816}}.

\bibitem[KP97]{kamps-porter:abstract-and-simple}
Klaus~Heiner Kamps and Timothy Porter, \emph{Abstract homotopy and simple
  homotopy theory}, World Scientific Publishing Co., Inc., River Edge, NJ,
  1997, \href {http://dx.doi.org/10.1142/9789812831989}
  {\path{doi:10.1142/9789812831989}}.

\bibitem[KS17]{kapulkin-szumilo:internal-language-lex}
Krzysztof Kapulkin and Karol Szumi{\l}o, \emph{Internal language of finitely
  complete $(\infty,1)$-categories}, preprint, 2017, \href
  {http://arxiv.org/abs/1709.09519} {\path{arXiv:1709.09519}}.

\bibitem[KSo17]{kapulkin-szumilo}
Krzysztof Kapulkin and Karol Szumi\l~o, \emph{Quasicategories of frames of
  cofibration categories}, Appl. Categ. Structures \textbf{25} (2017), no.~3,
  323--347, \href {http://dx.doi.org/10.1007/s10485-015-9422-y}
  {\path{doi:10.1007/s10485-015-9422-y}}.

\bibitem[Lum10]{lumsdaine:thesis}
Peter~LeF. Lumsdaine, \emph{Higher categories from type theories}, Ph.D.
  thesis, Carnegie Mellon University, 2010,
  \url{http://peterlefanulumsdaine.com/research/Lumsdaine-2010-Thesis.pdf}.

\bibitem[Lur09]{lurie:htt}
Jacob Lurie, \emph{Higher topos theory}, Annals of Mathematics Studies, vol.
  170, Princeton University Press, Princeton, NJ, 2009,
  \url{http://www.math.harvard.edu/~lurie/papers/croppedtopoi.pdf}.

\bibitem[Shu12]{shulman:internal-languages}
Michael Shulman, \emph{Internal languages for higher categories}, October 2012,
  Slides from talk at Institute for Advanced Study,
  \url{http://home.sandiego.edu/~shulman/papers/higheril.pdf}.

\bibitem[Shu15a]{shulman:elegant-presheaves}
\bysame, \emph{The univalence axiom for elegant {R}eedy presheaves}, Homology
  Homotopy Appl. \textbf{17} (2015), no.~2, 81--106, \href
  {http://arxiv.org/abs/1307.6248} {\path{arXiv:1307.6248}}, \href
  {http://dx.doi.org/10.4310/HHA.2015.v17.n2.a6}
  {\path{doi:10.4310/HHA.2015.v17.n2.a6}}.

\bibitem[Shu15b]{shulman:inverse-diagrams}
\bysame, \emph{Univalence for inverse diagrams and homotopy canonicity}, Math.
  Structures Comput. Sci. \textbf{25} (2015), no.~5, 1203--1277, \href
  {http://arxiv.org/abs/1203.3253} {\path{arXiv:1203.3253}}, \href
  {http://dx.doi.org/10.1017/S0960129514000565}
  {\path{doi:10.1017/S0960129514000565}}.

\bibitem[Shu17]{shulman:ei-diagrams}
\bysame, \emph{Univalence for inverse {EI} diagrams}, Homology Homotopy Appl.
  \textbf{19} (2017), no.~2, 219--249, \href {http://arxiv.org/abs/1508.02410}
  {\path{arXiv:1508.02410}}, \url{https://doi.org/10.4310/HHA.2017.v19.n2.a12}.

\bibitem[Spi01]{spitzweck:operads-algebras-modules}
Markus Spitzweck, \emph{Operads, algebras and modules in general model
  categories}, preprint, 2001, \href {http://arxiv.org/abs/math/0101102}
  {\path{arXiv:math/0101102}}.

\bibitem[Str91]{streicher:semantics-book}
Thomas Streicher, \emph{Semantics of type theory: Correctness, completeness and
  independence results}, Progress in Theoretical Computer Science,
  Birkh\"a{}user Boston Inc., Boston, MA, 1991.

\bibitem[Szu14]{szumilo:two-models}
Karol Szumi{\l}o, \emph{Two models for the homotopy theory of cocomplete
  homotopy theories}, Ph.D. thesis, University of Bonn, 2014, \href
  {http://arxiv.org/abs/1411.0303} {\path{arXiv:1411.0303}}.

\bibitem[To{\"e}05]{toen:unicity}
Bertrand To{\"e}n, \emph{Vers une axiomatisation de la th\'eorie des
  cat\'egories sup\'erieures}, $K$-Theory \textbf{34} (2005), no.~3, 233--263,
  \href {http://arxiv.org/abs/math/0409598} {\path{arXiv:math/0409598}}, \href
  {http://dx.doi.org/10.1007/s10977-005-4556-6}
  {\path{doi:10.1007/s10977-005-4556-6}}.

\bibitem[Ton13]{tonelli}
Simone Tonelli, \emph{Investigations into a model of type theory based on the
  concept of basic pair}, Master's thesis, Stockholm University, 2013,
  \url{http://kurser.math.su.se/pluginfile.php/16103/mod_folder/content/0/2013/2013_08_report.pdf}.

\bibitem[{Uni}13]{hott:book}
The {Univalent Foundations Program}, \emph{Homotopy type theory: Univalent
  foundations of mathematics}, \url{http://homotopytypetheory.org/book},
  Institute for Advanced Study, 2013.

\bibitem[Voe15a]{voevodsky:C-system-from-universe}
Vladimir Voevodsky, \emph{A {C}-system defined by a universe category}, Theory
  Appl. Categ. \textbf{30} (2015), No. 37, 1181--1215, \href
  {http://arxiv.org/abs/1409.7925} {\path{arXiv:1409.7925}}.

\bibitem[Voe15b]{voevodsky:identity-types}
\bysame, \emph{{M}artin-{L}{\"o}f identity types in the {C}-systems defined by
  a universe category}, preprint, 2015, \href {http://arxiv.org/abs/1505.06446}
  {\path{arXiv:1505.06446}}.

\bibitem[Voe16a]{voevodsky:products}
\bysame, \emph{Products of families of types and {$(\Pi,\lambda)$}-structures
  on {C}-systems}, Theory Appl. Categ. \textbf{31} (2016), Paper No. 36,
  1044--1094, \href {http://arxiv.org/abs/1503.07072}
  {\path{arXiv:1503.07072}}.

\bibitem[Voe16b]{voevodsky:quotients-C-system}
\bysame, \emph{Subsystems and regular quotients of {C}-systems}, A Panorama of
  Mathematics: Pure and Applied, Contemp. Math., vol. 658, Amer. Math. Soc.,
  Providence, RI, 2016, pp.~127--137, \href {http://arxiv.org/abs/1406.7413}
  {\path{arXiv:1406.7413}}, \href {http://dx.doi.org/10.1090/conm/658/13124}
  {\path{doi:10.1090/conm/658/13124}}.

\end{thebibliography}

\end{document}